\documentclass[12pt]{amsart}
\usepackage{etex}
\usepackage{amsfonts, amssymb, latexsym}
\usepackage{amsthm,xparse}

\usepackage{amscd,amssymb}
\usepackage{verbatim}
\usepackage{pstricks}
\usepackage{pst-node}
\usepackage[mathscr]{euscript}
\usepackage{tikz}
\usepackage{tikz-cd}
\usepackage{tkz-euclide}
\usetikzlibrary{matrix,arrows, calc,decorations.pathmorphing,shapes}
\usetikzlibrary{matrix,arrows,backgrounds,shapes.misc,shapes.geometric,patterns,calc,positioning}
\usetikzlibrary{chains,positioning,fit}
\usetikzlibrary{decorations.pathreplacing,calligraphy, decorations.pathmorphing}
\usetikzlibrary{arrows.meta}

\usepackage[normalem]{ulem}
\usepackage{varwidth}
\usepackage{leftidx}

\usepackage{wrapfig,caption}
\usepackage{xfp} 
\usepackage{soul}

\usepackage[colorlinks=true, pdfstartview=FitV, linkcolor=blue, citecolor=blue, urlcolor=blue]{hyperref}

\usepackage{xcolor, fancyhdr}

\definecolor{darkmode}{RGB}{32, 31, 30}
\definecolor{gblue}{RGB}{190,209,210}
\definecolor{oblue}{RGB}{114, 160, 193}
\definecolor{dartmouthgreen}{rgb}{0.05, 0.5, 0.06}

\newsymbol\pp 1275

\usepackage{fullpage}
\usepackage{graphicx}
\usepackage{enumerate}
\usepackage{expl3}

\def\VR{\kern-\arraycolsep\strut\vrule &\kern-\arraycolsep}
\def\vr{\kern-\arraycolsep & \kern-\arraycolsep}

\newtheorem{theorem}{Theorem}[section]

\newtheorem{lemma}[theorem]{Lemma}
\newtheorem{prop}[theorem]{Proposition}
\newtheorem{corollary}[theorem]{Corollary}

\newtheorem{problem}{Problem}

\theoremstyle{definition}
\newtheorem{definition}[theorem]{Definition}

\newtheorem{rmk}[theorem]{Remark}
\newenvironment{remark}[1][]{\begin{rmk}[#1]\pushQED{\qed}}{\popQED \end{rmk}}

\newtheorem{qu}[theorem]{Question}

\newtheorem*{rmknonum}{Remark}

\newtheorem{obs}[theorem]{Observation}

\newtheorem{ex}[theorem]{Example}
\newenvironment{example}[1][]{\begin{ex}[#1]\pushQED{\qed}}{\popQED \end{ex}}
\newcommand{\pctext}[2]{\text{\parbox{#1}{\centering #2}}}

\newcommand{\Hom}{\operatorname{Hom}}
\newcommand{\End}{\operatorname{End}}

\newcommand{\rep}{\operatorname{rep}}

\newcommand{\SI}{\operatorname{SI}}
\newcommand{\SL}{\operatorname{SL}}
\newcommand{\GL}{\operatorname{GL}}

\newcommand{\ZZ}{\mathbb Z}
\newcommand{\CC}{\mathbb C}
\newcommand{\RR}{\mathbb R}
\newcommand{\NN}{\mathbb N}

\newcommand{\I}{\mathcal I }

\newcommand{\eps}{\epsilon} 
\newcommand{\be}{\begin{enumerate}} 
\newcommand{\ee}{\end{enumerate}}
\newcommand{\dsp}{\displaystyle}
\newcommand{\hb}{\widehat{\beta}}

\newcommand{\V}{V}

\newcommand{\Id}{\mathbf{Id}}

\newcommand{\ddim}{\operatorname{\mathbf{dim}}}

\newcommand{\Q}{\mathbf{\mathcal{Q}}}

\newcommand{\E}{\epsilon}
\newcommand{\p}{\mathcal{P}}

\newcommand{\F}{\mathcal{F}}

\newcommand{\K}{\mathcal{K}}

\newcommand{\unlm}{\underline{\lambda}}
\newcommand{\lm}{\lambda}
\newcommand{\river}{\mathcal{T}}

\newcommand{\wsi}[1]{\widehat{\sigma}({#1})}

\newcommand{\Eff}{\operatorname{Eff}}

\DeclareMathOperator{\Herm}{\operatorname{Herm}}

\newcommand\restr[2]{{
  \left.\kern-\nulldelimiterspace 
  #1 
  \vphantom{\big|} 
  \right|_{#2} 
  }}

\tikzset{snake it/.style={decorate, decoration=snake}}

\usetikzlibrary{decorations.pathreplacing,
                calligraphy}
\tikzset{
B/.style = {decorate,
            decoration={calligraphic brace, amplitude=8pt,
            raise=15pt, mirror},
            very thick,
            pen colour=black},
dot/.style = {circle, fill, inner sep=2pt, outer sep=0pt}
        }

\usetikzlibrary{decorations.shapes}

\tikzset{decorate sep/.style 2 args=
{decorate,decoration={shape backgrounds,shape=circle,shape size=#1,shape sep=#2}}}

\usetikzlibrary{quotes}

\begin{document}

\title{Hive-type polytopes for quiver multiplicities and the membership problem for quiver moment cones}
\author{Calin Chindris}
\address{University of Missouri-Columbia, Mathematics Department, Columbia, MO, USA}
\email[Calin Chindris]{chindrisc@missouri.edu}

\author{Brett Collins}
\address{University of Georgia, Mathematics Department, Athens, GA, USA}
\email[Brett Collins]{brett.collins@uga.edu}

\author{Daniel Kline}
\address{College of the Ozarks, Mathematics Department, Point Lookout, MO, USA}
\email[Daniel Kline]{dkline@cofo.edu}

\date{\today}
\bibliographystyle{amsalpha}
\subjclass[2010]{16G20, 13A50, 14L24}
\keywords{Hive polytopes, Littlewood-Richardson coefficients, moment cones, strongly polynomial time, quiver exceptional sequences, quiver semi-invariants}

\begin{abstract} Let $\Q$ be a bipartite quiver with vertex set $\Q_0$ such that the number of arrows between any source vertex and any sink vertex is constant. Let $\beta=(\beta(x))_{x \in \Q_0}$ be a dimension vector of $\Q$ with positive integer coordinates. 

Let $\rep(\Q, \beta)$ be the representation space of $\beta$-dimensional representations of $\Q$ and $\GL(\beta)$ the base change group acting on $\rep(Q, \beta)$ be simultaneous conjugation. Let $K^{\beta}_{\unlm}$ be the multiplicity of the irreducible representation of $\GL(\beta)$ of highest weight $\unlm$ in the ring of polynomial functions on $\rep(\Q, \beta)$. 

We show that $K^{\beta}_{\unlm}$ can be expressed as the number of lattice points of a polytope obtained by gluing together two Knutson-Tao hive polytopes. Furthermore, this polytopal description together with Derksen-Weyman's Saturation Theorem for quiver semi-invariants allows us to use Tardos' algorithm to solve the membership problem for the moment cone associated to $(\Q,\beta)$ in strongly polynomial time.  
\end{abstract}

\maketitle
\setcounter{tocdepth}{1}
\tableofcontents

\section{Introduction} 
\label{intro-sec}
\subsection{Motivation}
The Littlewood-Richardson coefficients are fundamental structure constants in algebraic combinatorics, representation theory and other areas in mathematics, mathematical physics, and algebraic complexity theory. In \cite{KT}, Knutson and Tao found a beautiful polytopal description of the Littlewood-Richardson coefficients in terms of certain triangular arrays of numbers, known as hives (see also the exposition by Buch \cite{Buch-2000}). This description plays a crucial role in the (first) proof of the Saturation Conjecture of the Littlewood-Richardson coefficients. Furthermore, Mulmuley, Naranayan, and Sohoni \cite{MNS} used the Knutson-Tao hive model and the Saturation Property of the Littlewood-Richardson coefficients to test their positivity in strongly polynomial time.

In this paper we aim to find similar polytopal descriptions for the more general multiplicities $K^{\beta}_{\unlm}$ and provide applications to the membership problem for moment cones of quivers.

Let $Q$ be a general quiver with set of vertices $Q_0$ and set of arrows $Q_1$. For an arrow $a \in Q_1$, we denote its tail and head by $ta$ and $ha$, respectively. Let $\beta=(\beta(x))_{x \in Q_0} \in \ZZ_{>0}^{Q_0}$ be a sincere dimension vector of $Q$ and let us consider the representation space of $\beta$-dimensional representations of $Q$,
$$
\rep(Q, \beta):=\prod_{a \in Q_1} \CC^{\beta(ha) \times \beta(ta)}.
$$ 
The base change group $\GL(\beta):=\prod_{x \in Q_0} \GL(\beta(x))$ acts on $\rep(Q, \beta)$ by simultaneous conjugation. This action gives rise to a rational convex polyhedral cone (see \cite{Sj}), which we refer to as the moment cone associated to $(Q, \beta)$ (see also \cite{CC4}). It is defined as follows: 
\[
\Delta(Q, \beta):= \left \{(\lm(x))_{x \in Q_0} \left \vert \pctext{4.4in}{$\lm(x)$ is a weakly decreasing sequence of $\beta(x)$ real numbers \\ such that there exists $W \in \rep(Q, \beta)$ with $\lm(x)$ the spectrum  of $\sum_{\substack{\text{$a \in Q_1$}\\ \text{ $ta = x$}} } W(a)^* \cdot W(a) - \sum_{\substack{\text{$a \in Q_1$} \\ \text{$ha = x$}}} W(a) \cdot W(a)^* $ for all $x \in Q_0$} \right. \right \},
\] 
where $W(a)^* \in \CC^{\beta(ta)\times \beta(ha)}$ denotes the transpose of the conjugate of $W(a)$ for every $a \in Q_1$.

For example, consider $\Q=\bullet \rightarrow \bullet \leftarrow \bullet$ and $\beta =(r,r,r)$. In this case, the multiplicities $K^{\beta}_{\unlm}$ are the Littlewood-Richardson coefficients corresponding to triples of partitions of length at most $r$, and the moment cone $\Delta(\Q, \beta)$ is essentially the Klyachko cone (see Example \ref{Kly-cone-ex} for more details). 

\subsection{Our results}
In this paper, we focus our attention on bipartite quivers $\Q$ with $m$ source vertices, $l$ sink vertices, and $n$ arrows between any two source and sink vertices. We refer to such quivers as $n$-complete bipartite quivers. 
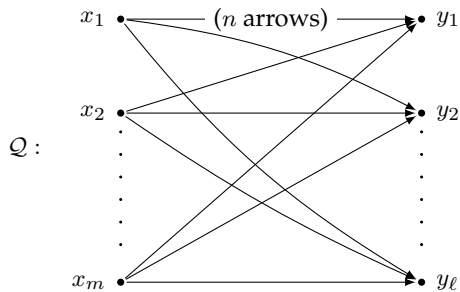
\begingroup\makeatletter\def\f@size{9.5}\check@mathfonts
$$
\Q:~
\vcenter{\hbox{  
\begin{tikzpicture}[[every edge quotes/.style={fill=white,font=\footnotesize}, point/.style={shape=circle, fill=black, scale=.3pt,outer sep=3pt},>=latex]
   \node[point,label={left:$x_1$}] (1) at (-4,1.75) {};
   \node[point,label={left:$x_2$}] (2) at (-4,.5) {};
   \node[point,label={left:$x_m$}] (3) at (-4,-1.75) {};
   
   \node[point,label={right:$y_1$}] (-1) at (0,1.75) {};
   \node[point,label={right:$y_2$}] (-2) at (0,.5) {};
   \node[point,label={right:$y_\ell$}] (-3) at (0,-1.75) {};
 \draw[decorate sep={.25mm}{3mm},fill] (0,.25) -- (0,-1.5);
  \draw[decorate sep={.25mm}{3mm},fill] (-4,.25) -- (-4,-1.5);
  
 \begin{scope}[every node/.style={fill=white}]
    \path[->]  
 (1) edge node{($n$ arrows)}	 (-1)
  (1) edge  [bend left=10] (-2)
  (1) edge  [bend right=10] (-3)
  (2) edge  (-1)
  (2) edge (-2)
  (2) edge [bend right=5]  (-3)
  (3) edge  (-1)
  (3) edge (-2)
  (3) edge  (-3);
  \end{scope}
\end{tikzpicture} 
}}
$$
\endgroup

Let $\unlm=(\lm(x_i), -\lm(y_j))_{i \in [m], j \in [\ell] }$ be a tuple of sequences with $\lm(x_i)$ a partition of length at most $\beta(x_i)$ and $\lm(y_j)$ a partition of length at most $\beta(y_j)$. Here, for $\lambda=(\lambda_1, \ldots, \lambda_N)$ a weakly decreasing sequence, $-\lambda$ denotes the weakly decreasing sequence $(-\lambda_N, \ldots, -\lambda_1)$. 

Let $K^{\beta}_{\unlm}$ be the multiplicity of the irreducible representation of $\GL(\beta)$ of highest weight $\unlm$ in $\CC[\rep(\Q,\beta)]$, the ring of polynomial functions on $\rep(\Q, \beta)$. We point out that the multiplicities $K^{\beta}_{\unlm}$ can also be expressed as dimensions\footnote{Any multiplicity $K^{\beta}_{\unlm}$ can be expressed as $\dim \SI(\Q_\beta, \widetilde{\beta})_{\widetilde{\sigma}}$ for a suitable weight $\widetilde{\sigma}$. This step is essential to our analysis because it allows us to use powerful methods from quiver invariant theory to derive the formula $(\ref{main-formula-mult})$. However, despite these advantages,  the explicit polytope we obtain in the first part of Theorem \ref{main-thm} cannot be constructed in strongly polynomial time when given a weight $\widetilde{\sigma}$ as input.  This limitation is the main obstacle to concluding that the positivity of $\dim \SI(\Q_\beta, \widetilde{\beta})_{\widetilde{\sigma}}$ can be decided in strongly polynomial time.  This problem, which we refer to as the generic quiver semi-stability problem, is still wide open.} of weight spaces of semi-invariants on the representation space $\rep(\Q_\beta, \widetilde{\beta})$, where $(\Q_\beta, \widetilde{\beta})$ is the flag-extension of $(\Q, \beta)$; see diagram $(\ref{flag-extension-pic})$ for details on how to draw $(\Q_\beta, \widetilde{\beta})$.  

Our main goal is to provide an explicit, polytopal description of the multiplicities $K^{\beta}_{\unlm}$. This description combined with Derksen-Weyman's Saturation Theorem (see \cite{DW1}) allows us to use Tardos' strongly polynomial time algorithm (see \cite{Tar86}) in our context. 

\begin{theorem}\label{main-thm} Let $\Q$ be an $n$-complete bipartite quiver with source vertices $x_1, \ldots, x_m$ and sink vertices $y_1, \ldots, y_\ell$ and let $\beta=(\beta(x))_{x \in \Q_0}$ be a sincere dimension vector of $\Q$. 

Let $\unlm=(\lm(x_i), -\lm(y_j))_{i \in [m], j \in [\ell] }$ be a tuple of sequences with $\lm(x_i)$ a partition of length at most $\beta(x_i)$ and $\lm(y_j)$ a partition of length at most $\beta(y_j)$ such that
$$
\sum_{i=1}^m |\lambda(x_i)|=\sum_{j=1}^{\ell} |\lambda(y_j)|.
$$

\begin{enumerate}[\normalfont(1)]
\item The multiplicity $K^{\beta}_{\unlm}$ can be expressed as the number of lattice points of a polytope $\p_{\unlm}$ obtained by gluing together two Knutson-Tao hive polytopes. 

\item There exists a strongly polynomial time algorithm to decide if $K^{\beta}_{\unlm}>0$. In particular, checking membership in the moment cone $\Delta(\Q, \beta)$ can be accomplished in strongly polynomial time.
\end{enumerate} 
\end{theorem}

To prove the first part of Theorem \ref{main-thm}, we establish in Theorem \ref{main-formula-mult-thm} a formula that expresses the multiplicity $K^{\beta}_{\unlm}$ as a sum of products of two multiple Littlewood-Richardson coefficients. This is achieved by first viewing $K^{\beta}_{\unlm}$ as the dimension of a weight space of semi-invariants for $\Q_\beta$ and then using quiver exceptional sequences and Derksen-Weyman's Embedding Theorem to embed $\Q_\beta$ into another quiver $\river$, introduced in Section \ref{qes-embedding-sec}. It is this new quiver $\river$ and its weight spaces of semi-invariants that enable us to derive the desired formula for $K^{\beta}_{\unlm}$ (see also Remark \ref{rmk-direct-compute-vs-river-quiver}). This formula leads us to the polytope $\p_{\unlm}$ that can be described as a combinatorial linear program and, furthermore, the positivity of $K^{\beta}_{\unlm}$ is equivalent to the feasibility of the corresponding combinatorial linear program (see Proposition \ref{prop-main-thm-part-1}). In Section \ref{moment-cones-proof-thm-1-sec}, we first show that a tuple $\underline{\lambda}$ of weakly decreasing sequences of integers lies in $\Delta(\Q, \beta)$ if and only if $K^{\beta}_{\unlm}$ is positive. Thus, checking membership in $\Delta(\Q, \beta)$ is equivalent to checking the feasibility of a combinatorial linear program that can be checked in strongly polynomial time via Tardos' algorithm.

In a recent paper \cite{Ver-Wal-2023}, Vergne and Walter generalized our Theorem \ref{main-thm} by proving the existence of polytopes that are less explicit than ours but they work for arbitrary acyclic quivers; see Remark \ref{rmk-abstract-vs-explicit-polytopes} for more details. This, combined with Tardos' algorithm, allowed them to conclude that the membership problem for moment cones for general acyclic quivers can be solved in strongly polynomial time. 

\section{Background on Quiver Invariant Theory}\label{QIT-semi-invariants-sec}

\subsection{Quivers and their representations} Throughout, we work over the field $\CC$ of complex numbers and denote by $\NN=\{0,1,\dots \}$. For a positive integer $L$, we denote by $[L]=\{1, \ldots, L\}$.

A quiver $Q=(Q_0,Q_1,t,h)$ consists of two finite sets $Q_0$ (vertices) and $Q_1$ (arrows) together with two maps $t:Q_1 \to Q_0$ (tail) and $h:Q_1 \to Q_0$ (head). We represent $Q$ as a directed graph with set of vertices $Q_0$ and directed edges $a:ta \to ha$ for every $a \in Q_1$.  A quiver is said to be acyclic if it has no oriented cycles. We call a quiver connected if its underlying graph is connected.

A representation of $Q$ is a family $V=(V(x), V(a))_{x \in Q_0, a\in Q_1}$, where $V(x)$ is a finite-dimensional $\CC$-vector space for every $x \in Q_0$, and $V(a): V(ta) \to V(ha)$ is a $\CC$-linear map for every $a \in Q_1$. After fixing bases for the vector spaces $\V(x)$, $x \in Q_0$, we often think of the linear maps $\V(a)$, $a \in Q_1$, as matrices of appropriate size. A subrepresentation $W$ of $V$, written as $W \subseteq V$, is a representation of $Q$ such that $W(x) \subseteq V(x)$  for every $x \in Q_0$, and moreover $V(a)(W(ta)) \subseteq W(ha)$ and $W(a)=\restr{V(a)}{W(ta)}$ for every arrow $a \in Q_1$. 

A morphism $\varphi:V \rightarrow W$ between two representations is a collection $(\varphi(x))_{x \in Q_0}$ of $\CC$-linear maps with $\varphi(x) \in \Hom_{\CC}(V(x), W(x))$ for every $x \in Q_0$, and such that $\varphi(ha) \circ V(a)=W(a) \circ \varphi(ta)$ for every $a \in Q_1$. The $\CC$-vector space of all morphisms from $V$ to $W$ is denoted by $\Hom_Q(V, W)$.

The dimension vector $\ddim V \in \NN^{Q_0}$ of a representation $V$  is defined by $\ddim V(x)=\dim_\CC V(x)$ for all $x \in Q_0$. By a dimension vector of $Q$, we simply mean an $\NN$-valued function on the set of vertices $Q_0$. We say a dimension vector $\beta$ is sincere if $\beta(x)>0$ for every $x \in Q_0$. For every vertex $x \in Q_0$, the simple dimension vector at $x$, denoted by $e_x$, is defined by $e_x(y)=\delta_{x,y}$, $\forall y \in Q_0$, where $\delta_{x,y}$ is the Kronecker symbol. We point out that $e_x$ is the dimension vector of the simple representation $S_x$ defined by assigning a copy of $\CC$ to vertex $x$, the zero vector space at all other vertices, and the zero linear map along all arrows. 

The Euler form (also known as the Ringel form) of $Q$ is the bilinear form on $\ZZ^{Q_0}$ defined by
$$
\langle \alpha, \beta \rangle:=\sum_{x \in Q_0}\alpha(x)\beta(x)-\sum_{a \in Q_1} \alpha(ta)\beta(ha), \; \forall \alpha, \beta \in \ZZ^{Q_0}.
$$

From now on, we assume that all of our quivers are connected and acyclic. Then, for any integral weight $\sigma \in \ZZ^{Q_0}$, there exists a unique $\alpha \in \ZZ^{Q_0}$ such that $\sigma(x)=\langle \alpha, e_x \rangle$, $\forall x \in Q_0$. 

\subsection{Weight spaces of semi-invariants and quiver semi-stability}

Let $\beta$ be a sincere dimension vector of a quiver $Q$. As mentioned in Section \ref{intro-sec}, there is a natural action via simultaneous conjugation of $\GL(\beta)$ on $\rep(Q,\beta)$, \emph{i.e.}, for $g=(g(x))_{x \in Q_0} \in \GL(\beta)$ and $W=(W(a))_{a \in Q_1} \in \rep(Q,\beta)$, we define $g \cdot W \in \rep(Q, \beta)$ by
\[
(g \cdot W)(a):=g(ha)\cdot W(a) \cdot g(ta)^{-1}, \; \forall a \in Q_1.
\]

\noindent
This action descends to that of the subgroup
\[
\SL(\beta) := \prod_{x \in Q_0} \SL(\beta(x)),
\]
giving rise to a highly non-trivial ring of semi-invariants $\SI(Q,\beta):= \CC[\rep(Q,\beta)]^{\SL(\beta)}$. (We point out that since $Q$ is assumed to be acyclic, the invariant ring $\CC[\rep(Q, \beta)]^{\GL(\beta)}$ is precisely $\CC$.) Since $\GL(\beta)$ is linearly reductive and $\SL(\beta)$ is its commutator subgroup, we have the weight space decomposition 
\[
\SI(Q,\beta) = \bigoplus_{\chi \in X^*(\GL(\beta))} \SI(Q,\beta)_{\chi},
\]
where $X^*(\GL(\beta))$ is the group of rational characters of $\GL(\beta)$ and 
\[
\SI(Q,\beta)_{\chi}:= \{f \in \CC[\rep(Q,\beta)] \mid g \cdot f = \chi(g)f, \, \forall g \in \GL(\beta)\}
\]
is the space of \emph{semi-invariants of weight $\chi$}. Every integral weight $\sigma \in \ZZ^{Q_0}$ defines a character $\chi_{\sigma}$ of $\GL(\beta)$  by $\chi_{\sigma}(g):=\prod_{x \in Q_0} (\det g(x))^{\sigma(x)}$, $\forall g=(g(x))_{x \in Q_0} \in \GL(\beta)$. Moreover, since $\beta$ is sincere, any character of $\GL(\beta)$ is of the form $\chi_{\sigma}$ for a unique $\sigma \in \ZZ^{Q_0}$, allowing us to identify the character group with $\ZZ^{Q_0}$.  In what follows, we write $\SI(Q, \beta)_{\sigma}$ for $\SI(Q, \beta)_{\chi_\sigma}$.

In \cite{K}, King  used weight spaces of semi-invariants and tools from Geometric Invariant Theory to construct moduli spaces of quiver representations. Our focus in this paper is on combinatorial/computational aspects of weight spaces of semi-invariants.  

\begin{problem}[\textbf{The Polytopal Problem for quiver semi-invariants}]\label{poly-prob-quiver-semi-inv}
Let $Q$ be a quiver, $\beta$ a sincere dimension vector of $Q$, and $\sigma$ an integral weight of $Q$ such that $\sigma \cdot \beta=0$. Find an explicit rational polytope $\p_{\sigma}$ such that
\begin{enumerate}[\normalfont(1)]
\item $\dim \SI(Q, \beta)_\sigma=\text{the number of lattice points of~} \p_{\sigma}$;

\item $\p_{\sigma}$ can be described by a combinatorial linear program $A\mathbf{x} \leq \mathbf{b}$, where $A$ does not depend on $\sigma$, and the coordinates of $\mathbf{b}$ are homogeneous linear forms in the coordinates of $\sigma$. (This latter condition implies that $r\p_\sigma=\p_{r \sigma}$ for any positive integer $r$.)
\end{enumerate}
\end{problem}
 
\noindent
The polytopal problem for quiver semi-invariants, where the emphasis is on explicit, combinatorial polytopes,  seems to be very difficult in general. There are only a few explicit examples of quivers in the literature where Problem \ref{poly-prob-quiver-semi-inv} has been solved; see \cite{CDW}, \cite{CC1}, \cite{CC2}, \cite{Col20}, and \cite{DW1}. All of these examples rely on Knutson-Tao's hive model for Littlewood-Richardson coefficients. In this paper, we solve Problem \ref{poly-prob-quiver-semi-inv} for $n$-complete bipartite quivers and their flag-extensions by using quiver exceptional sequences to embed these quivers into other quivers and then computing the dimensions of the weight spaces  of semi-invariants for those quivers (see the quiver $\river$ defined in Section \ref{qes-embedding-sec}). Directly computing dimensions of weight spaces of semi-invariants for these quivers without embedding leads to very complicated formulas (see Remark \ref{rmk-direct-compute-vs-river-quiver}).

\begin{rmk}\label{rmk-abstract-vs-explicit-polytopes}
As a ``\emph{straightforward variant of their [our] construction}", Vergne and Walter introduced in \cite{Ver-Wal-2023} polytopes whose numbers of lattice points are the dimensions of weight spaces of semi-invariants for general acyclic quivers. While it seems difficult to find explicit, geometric descriptions of these polytopes, their existence allows the authors of \emph{loc. cit.} to prove that the membership problem for $\Delta(Q, \beta)$ can be solved in strongly polynomial time for acyclic quivers $Q$.

On the other hand, in the general context of Geometric Complexity Theory, given a decision problem, it is not enough to find strongly polynomial time algorithms that are just efficient in theory (see \cite[pages 106-107]{MNS}). It is important to find simple, combinatorial algorithms that run in strongly polynomial time and do not depend on linear programming (or other complicated numerical procedures). Our polytopes $\p_{\unlm}$, available for $n$-complete bipartite quivers $\Q$, are explicit and can be geometrically visualized (see $(\ref{our-polytope-pic})$ and Definition \ref{defn-our-polytope}). This opens up the possibility of finding algorithms to test membership in $\Delta(\Q, \beta)$ in the same vein as the max-flow polynomial time algorithm found by B{\"u}rgisser and Ikenmeyer in \cite{BI}. 
\end{rmk}

The notion of a semi-stable quiver representation, introduced by King \cite{K} in the context of moduli spaces of quiver representations, plays a key role in understanding the positivity of the dimensions of weight spaces of semi-invariants. 

Let $\sigma \in \ZZ^{Q_0}$ be an integral weight of $Q$. A representation $W$ of $Q$ is $\sigma$-semi-stable if and only if the following conditions hold:
\begin{equation} \label{semi-stab-rep}
\sigma \cdot \ddim W=0 \text{~and~} \sigma \cdot \ddim(W')\leq 0, \; \forall \, W' \subseteq W.
\end{equation}

Let $\beta'$ be a dimension vector of $Q$ with $\beta'\leq \beta$, \emph{i.e.}, $\beta'(x) \leq \beta(x)$, $\forall x \in Q_0$. In what follows, we write $\beta' \hookrightarrow \beta$ to mean that a generic (equivalently, every) $\beta$-dimensional representation has a subrepresentation of dimension vector $\beta'$.  

\begin{example}\label{dim-sub-vectors-sinks-sources} If $x$ is a sink vertex of $Q$, it is immediate to see that any $\beta$-dimensional representation has the simple representation $S_x$ as a subrepresentation, and thus $e_x \hookrightarrow \beta$. On the other hand, if $x$ is a source vertex of $Q$, one can also easily see that $\beta-e_x \hookrightarrow \beta$.
\end{example}

The next fundamental result gives necessary and sufficient conditions for the positivity of $\dim \SI(Q, \beta)_{\sigma}$.

\begin{theorem}\label{King-criterion} For an integral weight $\sigma \in \ZZ^{Q_0}$ of $Q$, the following statements are equivalent:
\begin{enumerate}[\normalfont(1)]
\item $\dim \SI(Q, \beta)_\sigma>0$;

\item $\sigma \cdot \beta=0$ and $\sigma \cdot \beta' \leq 0$ for all $\beta' \hookrightarrow \beta$;

\item there exists a $\sigma$-semi-stable $\beta$-dimensional representation of $Q$;

\item there exists $W \in \rep(Q, \beta)$ such that 
	\[
	\sum_{\substack{a \in Q_1 \\ ta=x}} W(a)^* \cdot W(a) - \sum_{\substack{a \in Q_1 \\ ha=x}} W(a) \cdot W(a)^* = \sigma(x) \cdot \Id_{\beta(x)} \; \forall x \in Q_0.
	\]
\end{enumerate}
Consequently, weight spaces of quiver semi-invariants have the following \textbf{Saturation Property}:
$$
\dim \SI(Q, \beta)_{r\sigma}>0 \text{~for some positive integer~} r \geq 1 \text{~implies that~} \dim \SI(Q, \beta)_{\sigma}>0.
$$
\end{theorem}

The equivalence of $(1)$ and $(2)$, and the Saturation Property of quiver semi-invariants are due to Derksen and Weyman \cite{DW1} (see also \cite{CB}). The equivalence of $(2)$, $(3)$, and $(4)$ is due to King \cite{K}.

\begin{rmk} \phantomsection\label{wt-sign-at-sources-sinks} 
\begin{enumerate}
\item We point out that if $\dim \SI(Q, \beta)_{\sigma}>0$, then $\dim \SI(Q, \beta)_{r\sigma}>0$ for any positive integer $r$. Indeed, if $f \in \SI(Q, \beta)_\sigma$ is a non-zero semi-invariant then $f^r$ is a non-zero semi-invariant of weight $r\sigma$.

\item Assume that $\dim \SI(Q, \beta)_\sigma>0$. Then it follows from Theorem \ref{King-criterion} and Remark \ref{dim-sub-vectors-sinks-sources} that 
$$
\sigma(x) \geq 0 \text{~for any source vertex $x$,} \text{~and~} \sigma(y) \leq 0 \text{~for any sink vertex $y$}.
$$
\end{enumerate}
\end{rmk}

We recall another important result \cite[Lemma 6.5.7]{IOTW} (see also \cite[Lemma 3]{ChiGra-2019}) that gives necessary conditions for the positivity of $\dim \SI(Q,\beta)_{\sigma}$. It comes in handy in the proof of our main result,  Theorem \ref{main-formula-mult-thm}.

\begin{prop}\label{wt-dim-vector-prop} Let $\sigma \in \ZZ^{Q_0}$ be an integral weight of $Q$ with $\sigma=\langle \alpha, \cdot \rangle$ for a unique $\alpha \in \ZZ^{Q_0}$. If $\dim \SI(Q, \beta)_{\sigma}>0$ then $\alpha$ must be a dimension vector of $Q$, \emph{i.e.}, $\alpha(x) \geq 0$, $\forall x \in Q_0$.
\end{prop}

\begin{rmk} If $\beta$ is not sincere, then the positivity of $\dim \SI(Q, \beta)_\sigma$ does not necessarily imply that all the coordinates of $\alpha$ are non-negative. 
\end{rmk}
 
\subsection{The cone of effective weights}

Let $Q$ be a quiver and $\beta$ a sincere dimension vector of $Q$. The cone of effective weights associated to $(Q, \beta)$ is the rational convex polyhedral cone defined by
$$
\Eff(Q, \beta):=\{\sigma \in \RR^{Q_0} \mid \sigma \cdot \beta=0 \text{~and~}\sigma \cdot \beta'\leq 0, \; \forall \beta' \hookrightarrow \beta\}.
$$
It follows from Theorem \ref{King-criterion} that the lattice points of $\Eff(Q, \beta)$ is the affine semi-group of all integral weights $\sigma \in \ZZ^{Q_0}$ for which $\dim \SI(Q, \beta)_\sigma >0$. This is further equivalent to saying that there exists a $\beta$-dimensional $\sigma$-semi-stable representation. For further details, we refer the reader to \cite{DW1, DW2} and \cite{SVB}.

\begin{problem}[\textbf{The generic quiver semi-stability problem}]\label{gen-semi-stab-problem}
Let $Q$ be a quiver, $\beta$ a sincere dimension vector of $Q$, and $\sigma$ an integral weight of $Q$ such that $\sigma \cdot \beta=0$. Decide whether $\sigma$ belongs to $\Eff(Q, \beta)$.
\end{problem}

\begin{rmk} \label{generic-semi-problem-rmk} The Saturation Property for quiver semi-invariants tells us that for a given $\sigma \in \ZZ^{Q_0}$,
$$
\sigma \in \Eff(Q, \beta) \Longleftrightarrow \dim \SI(Q, \beta)_{\sigma} \neq 0.
$$
Thus, one might hope that a solution to  the Polytopal Problem \ref{poly-prob-quiver-semi-inv} combined with Tardos' algorithm would imply an effective solution to the generic quiver semi-stability Problem \ref{gen-semi-stab-problem}.  This is indeed the case assuming that the input is specified as in Remark \ref{rmk-non-natural-input-form}.

On the other hand,  when the input ($\beta$ and $\sigma$) is specified as lists of integers, each of length $|Q_0|$, we are not aware of any examples of polytopes $\p_\sigma$ (solutions to the Polytopal Problem) that can be constructed in strongly polynomial time from this input. We are thankful to M. Vergne and M. Walter for pointing this out to us. 
\end{rmk}

For the remainder of this section we assume that $Q$ is a bipartite quiver (not necessarily $n$-complete) with source vertices $x_1, \ldots, x_m$, and sink vertices $y_1, \ldots, y_\ell$. For a sincere dimension vector $\beta$, let $Q_\beta$ be the flag extension of $Q$ defined as below, where the flag $\F(x)$ is an equioriented type $\mathbb{A}$ quiver with $\beta(x) -1$ arrows for each $x \in Q_0$. We use \tikz   \draw [line width=1pt, double distance=2pt,
             arrows = {-Latex[length=0pt 3 .5]}] (0,0) -- (1,0);  to indicate that multiple arrows are allowed between vertices but  $Q$ need not be $n$-complete.
             
             \black

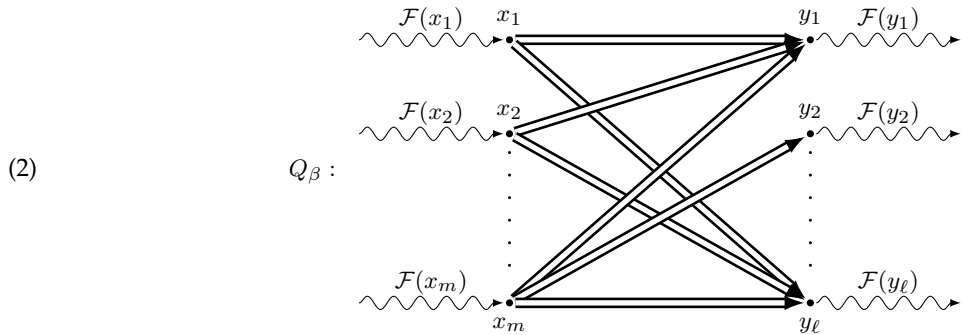
\begingroup\makeatletter\def\f@size{9.5}\check@mathfonts
\begin{equation}\label{flag-extension-pic}
Q_\beta:~
\vcenter{\hbox{  
\begin{tikzpicture}[point/.style={shape=circle, fill=black, scale=.3pt,outer sep=3pt},>=latex]
   \node[point,label={above:$x_1$}] (1) at (-4,1.75) {};
   \node[point,label={above:$x_2$}] (2) at (-4,.5) {};
   \node[point,label={below:$x_m$}] (3) at (-4,-1.75) {};
   
   \node[point,label={above:$y_1$}] (-1) at (0,1.75) {};
   \node[point,label={above:$y_2$}] (-2) at (0,.5) {};
   \node[point,label={below:$y_\ell$}] (-3) at (0,-1.75) {};
  
 \draw[decorate sep={.25mm}{3mm},fill] (0,.25) -- (0,-1.5);
  \draw[decorate sep={.25mm}{3mm},fill] (-4,.25) -- (-4,-1.5);

     \draw [line width=1pt, double distance=2pt,
             arrows = {-Latex[length=0pt 3 .5]}] (1) -- (-1);
        \draw [line width=1pt, double distance=2pt,
             arrows = {-Latex[length=0pt 3 .5]}] (1) -- (-3);
             
             \draw [line width=1pt, double distance=2pt,
             arrows = {-Latex[length=0pt 3 .5]}] (2) -- (-1);
             
                \draw [line width=1pt, double distance=2pt,
             arrows = {-Latex[length=0pt 3 .5]}] (2) -- (-3);
             
              \draw [line width=1pt, double distance=2pt,
             arrows = {-Latex[length=0pt 3 .5]}] (3) -- (-1);
             
              \draw [line width=1pt, double distance=2pt,
             arrows = {-Latex[length=0pt 3 .5]}] (3) -- (-2);
             
              \draw [line width=1pt, double distance=2pt,
             arrows = {-Latex[length=0pt 3 .5]}] (3) -- (-3);

  
    \path[draw, ->, snake it]  (-6,1.75)--(1) node[midway,above] {$\F(x_1)$};
    \path[draw, ->, snake it]  (-6,.5)--(2) node[midway,above] {$\F(x_2)$};
    \path[draw, ->, snake it]  (-6,-1.75)--(3) node[midway,above] {$\F(x_m)$};
    \path[draw, ->, snake it]  (-1)--(2,1.75) node[midway,above] {$\F(y_1)$};    
    \path[draw, ->, snake it]  (-2)--(2,.5) node[midway,above] {$\F(y_2)$};
    \path[draw, ->, snake it]  (-3)--(2,-1.75) node[midway,above] {$\F(y_{\ell})$};
\end{tikzpicture} 
}}
\end{equation}
\endgroup

We define $\widetilde{\beta}$ to be the extension of $\beta$ to $Q_{\beta}$ that takes values $1, \ldots, \beta(x_i)$ along the vertices (from left to right) of the flag $\F(x_i)$, $i \in [m]$, and $\beta(y_j), \ldots, 1$ along the vertices (from left to right) of the flag $\F(y_j)$, $j \in [\ell]$.

\begin{lemma}\label{eff-wts-cone-flags-lemma} Let $Q$ be a bipartite quiver with source vertices $x_1, \ldots, x_m$, and sink vertices $y_1, \ldots, y_\ell$, and let $\beta$ be a sincere dimension vector of $Q$. If $\widetilde{\sigma} \in \Eff(Q_\beta, \widetilde{\beta})$ is an effective weight, then
$$
\restr{\widetilde{\sigma}}{\F(x_i)}\geq 0, \forall i \in [m], \text{~and~} \restr{\widetilde{\sigma}}{\F(y_j)}\leq 0, \forall j \in [\ell].
$$
\end{lemma}

\begin{proof} We already know that $\widetilde{\sigma}$ is non-negative at the $m$ source vertices of $Q_\beta$ and non-positive at the $l$ sink vertices of $Q_\beta$ by Remark \ref{wt-sign-at-sources-sinks}.

Now let $W \in \rep(Q_\beta, \widetilde{\beta})$ be a generic representation such that $W(a)$ is injective along any given arrow $a$ of a flag $\F(x_i)$ and $W(b)$ is surjective along any given arrow $b$ of a flag $\F(y_j)$.  Since $\widetilde{\beta}(ha)=\widetilde{\beta}(ta)+1$ and $\widetilde{\beta}(tb)=\widetilde{\beta}(hb)+1$, it is immediate to see that $W$ has subrepresentations $W'_1$ and $W'_2$ of dimension vector $\widetilde{\beta}-e_{ha}$ and $e_{tb}$, respectively, where $W'_1$ is the same as $W$ except that at vertex $ha$ where $W'_1$ is the $(\beta(ha)-1)$-dimensional image of $W(a)$, and $W'_2$ is zero everywhere except at vertex $tb$ where $W'_2$ is the one-dimensional kernel of $W(b)$. 

The argument above shows that if $z$ is a non-source vertex of  $Q_\beta$ lying along one of the flags $\F(x_i)$, then $\widetilde{\beta}-e_z \hookrightarrow \widetilde{\beta}$ and thus $\widetilde{\sigma}(z) \geq 0$. Furthermore, if $z$ is a non-sink vertex of $Q_\beta$ lying along one of the flags $\F(y_j)$, then $e_z \hookrightarrow \widetilde{\beta}$ and thus $\widetilde{\sigma}(z) \geq 0$. This now completes the proof.
\end{proof}

\begin{rmk} 
\begin{enumerate}
\item As hinted in Theorem \ref{King-criterion}, there is a tight relationship between the moment cone $\Delta(Q, \beta)$ and the cone of effective weights $\Eff(Q_\beta, \widetilde{\beta})$; see Proposition \ref{moment-cone-effecive-cone-prop} for full details.

\item Let $\sigma$ be an integral weight of $Q$ and let $\sigma'$ be its trivial extension to $Q_\beta$ defined to be zero at all other vertices of $Q_\beta$. Then one can check that 
$$
\SI(Q, \beta)_\sigma=\SI(Q_{\beta}, \widetilde{\beta})_{\sigma'}.
$$
\end{enumerate}
\end{rmk}

\section{Quiver exceptional sequences and the Embedding Theorem for quiver semi-invariants} \label{qes-embedding-sec}
In this section, we first review Derksen-Weyman's  Embedding Theorem for quiver semi-invariants. This result allows us to embed the quiver $\Q_{\beta}$ into a new quiver, denoted below by $\river$, without changing the dimensions of the weight spaces of semi-invariants for $\Q_\beta$. The advantage of working with $\river$ is that it is significantly easier to find a polytopal description for the dimensions of its spaces of semi-invariants than for those of $\Q_\beta$ (see Sections \ref{compute-semi-inv-river-sec} - \ref{our-polytope-sec}).

In what follows, by a Schur representation $V$ of a quiver $Q$, we mean a representation such that $\dim \End_Q(V)=1$, i.e., $\End_Q(V)=\{(\lambda\Id_{V(x)})_{x \in Q_0} \mid \lambda \in \CC \}$. Furthermore, for two dimension vectors $\alpha$ and $\beta$, we define $(\alpha \circ \beta)_Q:=\dim \SI(Q, \beta)_{\langle \alpha, \cdot \rangle}$. (Whenever the quiver is understood from the context, we drop the subscript $Q$ and simply write $\alpha \circ \beta$ for the dimension of $\SI(Q, \beta)_{\langle \alpha, \cdot \rangle}$.) 

\begin{definition}[\textbf{Quiver Exceptional Sequences}] Let $Q=(Q_0, Q_1, t, h)$ be a quiver. A sequence $\E = (\eps_1, \ldots \eps_N)$ of dimension vectors is said to be a \emph{quiver exceptional sequence} if: 
	\be
		\item each $\eps_i$ is a real Schur root, \emph{i.e.},  $\langle \eps_i, \eps_i \rangle = 1$ and $\eps_i$ is the dimension vector of a Schur representation for all $i \in [N]$;
		
	    \item $\langle \eps_i, \eps_j \rangle \leq 0$ and $\eps_j \circ \eps_i\neq 0$ for all $1 \leq i < j \leq N$. 
	\ee
\end{definition} 

\begin{remark} To check the second condition in the definition above, we will use the following fact which is a consequence of Derksen-Weyman's First Fundamental Theorem for quiver semi-invariants \cite{DW1} (see also \cite{CB}). For two dimension vectors $\alpha$ and $\beta$ of $Q$, we have that $\alpha \circ \beta \neq 0$ if and only if
$$
\langle \alpha, \beta \rangle =0 \text{~and~} \Hom_Q(V, W)=0
$$ 
for some representations $V$ and $W$ of dimension vectors $\alpha$ and $\beta$, respectively.
\end{remark}

To any quiver exceptional sequence $\E = (\eps_1, \ldots, \eps_N)$, we associate the quiver $Q(\eps)$ with vertices $1, \ldots, N$ and $-\langle \eps_i, \eps_j \rangle$ arrows from vertices $i$ to $j$ for all $1 \leq i \neq j \leq N$. Let
\[ \I \colon \RR^N \longrightarrow \RR^{Q_0}\] 
be the map defined by 
\[ \I\left(\gamma(1), \ldots, \gamma(N) \right) \colon = \sum_{i=1}^{N} \gamma(i) \eps_i \; \; \text{for all } \; \gamma = \left (\gamma(1), \ldots, \gamma(N) \right) \in \RR^N.\]

We are now ready to state Derksen-Weyman's Embedding Theorem which plays a key role in our approach to computing the dimensions of weight spaces of quiver semi-invariants.

\begin{theorem}[\textbf{The Embedding Theorem for Quiver Semi-Invariants}] \cite{DW2}
\label{thm:embedding-semi-invariants}
Let $Q=(Q_0, Q_1, t, h)$ be a quiver and $\E = \left (\eps_1, \ldots, \eps_N \right)$ a quiver exceptional sequence. If $\alpha$ and $\beta$ are two dimension vectors of $Q(\eps)$, then
\[ (\alpha \circ \beta )_{Q(\eps)} = \left ( \I(\alpha) \circ \I(\beta) \right )_Q.\] 
\end{theorem} 

\smallskip
We end this section with an important example. Let $\Q$ be the $n$-complete bipartite quiver with source vertices $x_1, \ldots, x_m$,  sink vertices $y_1, \ldots, y_{\ell}$, and $n$ arrows from $x_i$ to $y_j$ for every $i \in [m]$ and $j \in [\ell]$. Let $\beta$ be a sincere dimension vector of $\Q$ and let $\Q_{\beta}$ be the corresponding flag-extension of $\Q$. In what follows, we show how to realize $\Q_{\beta}$ as $\river(\E)$ for a suitable quiver $\river$ and quiver exceptional sequence $\E$. 

\bigskip
Let $\river$ be the quiver defined as:
\smallskip
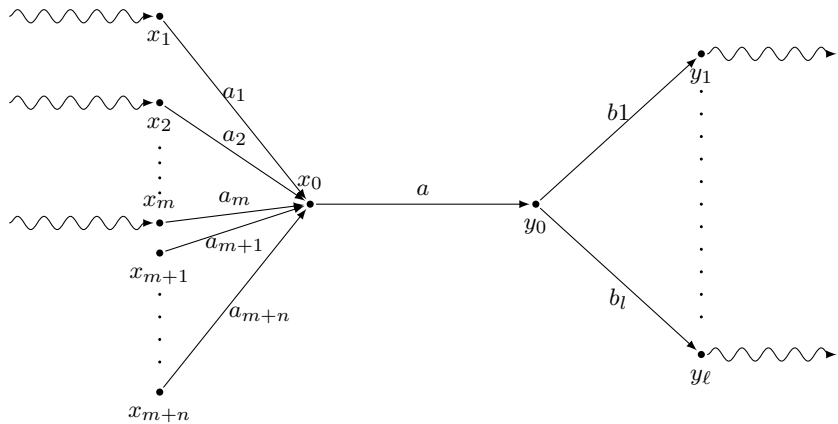
\begingroup\makeatletter\def\f@size{9.5}\check@mathfonts
$$
\vcenter{\hbox{  
\begin{tikzpicture}[point/.style={shape=circle, fill=black, scale=.3pt,outer sep=3pt},>=latex, decoration=snake]
   \node[point,label={below:$y_0$}] (0) at (1.5,0) {};
   \node[point,label={above:$x_0$}] (1) at (-1.5,0) {};

    \node[point,label={below:$y_1$}] (-1) at (3.7,2) {};
   \node[point,label={below:$y_{\ell}$}](-2) at (3.7,-2) {};
       
       \node[point,label={below:$x_{1}$}](2)  at (-3.5,2.5) {};
        \node[point,label={below:$x_{2}$}](3) at (-3.5,1.35) {};
      \node[point,label={above:$x_{m}$}](4) at (-3.5, -.25) {};
       \node[point,label={below:$x_{m+1}$}](5) at (-3.5, -.65) {};
     \node[point,label={below:$x_{m+n}$}](6) at (-3.5, - 2.5){};

     \draw[decorate sep={.25mm}{3mm},fill] (3.7,1.5)--(3.7,-1.5);
       \draw[decorate sep={.25mm}{2mm},fill] (-3.5,.75)--(-3.5,.1);   
       \draw[decorate sep={.25mm}{3mm},fill] (-3.5,-1.2)--(-3.5,-2.1);

    \path[->]  
 (0) edge node[above] {$b1$} (-1)  
  (0) edge  node[below] {$b_l$}(-2) 
  (1) edge node[above] {$a$} (0)  
   (2) edge  node[above] {$a_1$}  (1) 
    (3) edge node[above] {$a_2$}  (1) 
     (4) edge node[above] {$a_m$}   (1) 
      (5) edge node[below] {$a_{m+1}$}   (1) 
      (6) edge  node[below, xshift=1em] {$a_{m+n}$}  (1);
   
   \path[draw, ->, snake it]  (-1) -- (5.5, 2); 
    \path[draw, ->, snake it]   (-2) -- (5.5, -2); 
       \path[draw, ->, snake it]  (-5.5, 2.5) -- (2) ; 
         \path[draw, ->, snake it]  (-5.5, 1.35)-- (3); 
         \path[draw, ->, snake it]   (-5.5, -.25) -- (4) ; 

\end{tikzpicture} 
}}
$$
\endgroup

\smallskip
\noindent
with the flag $\F(x_i)$ going in vertex $x_i$ of length $\beta(x_i)-1$, $\forall i \in [m]$, and the flag $\F(y_j)$ going out of vertex $y_j$ of length $\beta(y_j)-1$, $\forall j \in [\ell]$. Note that there are no flags attached to the $n$ vertices $x_{m+1}, \ldots, x_{m+n}$. Also, if $\beta$ takes value one at a vertex of $\Q$, then no flag is attached to that vertex in $\river$.

Next, consider the dimension vectors $\delta_1, \ldots, \delta_m$ of $\river$ defined by
$$
\delta_i(x_0)=n+1, \delta_i(y_0)=n, \delta_i(x_i)=\delta_i(x_{m+1})=\ldots=\delta_i(x_{m+n})=1,
$$
and $\delta_i$ is zero at all other vertices of $\river$. To build the desired quiver exceptional sequence, we will work with the following dimension vectors: 
\begin{itemize}
\item the simple roots at the vertices of the flag $\F(x_i) \setminus \{x_i\}, i \in [m]$; 
\item $\delta_1, \ldots, \delta_m$;
\item the simple roots at the vertices of the flag $\F(y_j), j \in [\ell]$. 
\end{itemize}

\begin{prop} \label{prop-embed-qes} The dimension vectors above can be ordered to form a quiver exceptional sequence $\E$ for $\river$ such that $\river(\E) = \Q_{\beta}$.
\end{prop} 
\begin{proof} To obtain the sequence $\E$, we list the simple roots at the vertices of the flags $\F(x_1)\setminus \{x_1\},$ $\ldots,\F(x_m) \setminus \{x_m\}$ by going through the vertices of each flag from left to right starting with the flag $\F(x_1)$. Next, we list the dimension vectors $\delta_1, \ldots, \delta_m$. Finally,  we list the simple roots at the vertices of the flags $\F(y_1), \ldots, \F(y_\ell)$ by going through the vertices of each flag from left to right starting with the flag $\F(y_1)$. 

It is clear that any simple root is a real Schur root. Next, we show that the $\delta_i$ are real Schur roots and $\delta_i \perp \delta_j=0$ for all $i,j \in [m]$. For each $i \in [m]$,  consider the representation $V_i$ of $\river$ defined by 
\begin{itemize}
\item $V_i(x_0) = \CC^{n+1}, \; V_i(y_0) = \CC^n, \; V_i(x_i) =V_i(x_{m+1})= \ldots = V_i(x_{m+n}) = \CC$, and $V$ is zero at the remaining vertices;
\item $V_i(a): \CC^{n+1} \to \CC^n$ sends $(t_1,\ldots, t_{n+1})$ to $(t_1+t_{n+1},\ldots, t_n+t_{n+1})$;
\item $V_i(a_i): \CC \to \CC^{n+1}$ is the $(n+1)^{th}$ canonical inclusion of $\CC$ into $\CC^{n+1}$;
\item $V_i(a_{m+k}):\CC \to \CC^{n+1}$ is the $k^{th}$ canonical inclusion of $\CC$ into $\CC^{n+1}$ for every $k \in [n]$.
\end{itemize}

It is immediate to check that $\End_{\river}(V_i) = \{\lambda \Id_{V_i} \mid \lambda \in \CC\}$, \textit{i.e.}, $V_i$ is a Schur representation of dimension vector $\delta_i$ which together with the fact that $\langle \delta_i, \delta_i \rangle =1$ proves that $\delta_i$ is a real Schur root for all $i \in [m]$. Also, we have that $\Hom_{\river}(V_i,V_j) = 0$ and $\langle \delta_i, \delta_j \rangle = 0$, which imply that $\delta_i \circ \delta_j\neq 0$ for all $1 \leq i \neq j \leq m$. Thus, it is now clear that $\E$ is a quiver exceptional sequence with $\river(\E) = \Q_\beta$.
\end{proof}

\begin{example} In what follows, for two quivers $Q'$ and $Q$, we write $Q' \hookrightarrow Q$ to mean that $Q'=Q(\E)$ for an explicit quiver exceptional sequence $\E$.

\begin{enumerate}
\item  (\textbf{$n$-Kronecker quivers})

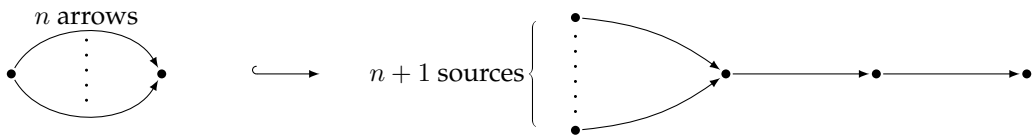
\[
\begin{tikzpicture}[point/.style={shape=circle, fill=black, scale=.3pt,outer sep=3pt},>=latex]
   \node[point,label={left:$ $}] (1) at (-2.5,0) {};
   \node[point,label={right:$ $}] (4) at (-.5,0) {};

        \draw[decorate sep={.25mm}{2mm},fill] (-1.5,.45)--(-1.5,-.5);

      \draw[right hook->] (.7,0) -- (1.6, 0);

    \path[->]  
  (1) edge  [bend left=60] (4);
  
      \path[->]  
  (1) edge  [bend right=60] (4);
  

    \node at (-1.5,.75) { \footnotesize $n$ arrows}; 
    
   
    \node[point,label={left:$ $}] (5) at (5,.75) {};
   
   \node[point,label={left:$ $}] (7) at (5,-.75) {};
   \node[point,label={right:$ $}] (8) at (7,0) {};
     \node[point,label={right:$ $}] (9) at (9,0) {};
   \node[point,label={left:$ $}] (12) at (11,0) {};

            \draw[decorate sep={.25mm}{2mm},fill] (5,.5)--(5,-.5);   
            
    \path[->]  
  (5) edge   [bend left=15] (8)
  (7) edge  [bend right=15]  (8)
  (8) edge  (9)
  (9) edge  (12); 
  
     \draw [decorate,
    decoration = {calligraphic brace, mirror, raise=15pt}] (5,.7) --  (5,-.7)
    node[pos=0.5,left=15pt,black]{\footnotesize $n+1$ sources}; 

\end{tikzpicture}  \] 

\item (\textbf{complete bipartite quivers})

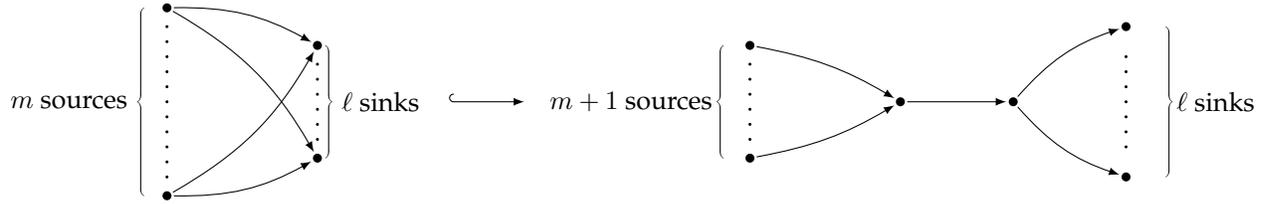
\[\begin{tikzpicture}[point/.style={shape=circle, fill=black, scale=.3pt,outer sep=3pt},>=latex]
   \node[point,label={left:$ $}] (1) at (-2.5,1.25) {};
   \node[point,label={left:$ $}] (2) at (-2.5,-1.25) {};
   \node[point,label={right:$ $}] (3) at (-.5,.75) {}; 
   \node[point,label={left:$ $}] (4) at (-.5,-.75) {};

        \draw[decorate sep={.25mm}{2mm},fill] (-2.5, 1)--(-2.5,-1.1);
         \draw[decorate sep={.25mm}{2mm},fill] (-.5, .5)--(-.5,-.5);

     \path[->]  
  (1) edge  [bend left=15] (3)
  (1) edge [bend left=15] (4)
    (2) edge  [bend right=15] (3)
  (2) edge [bend right=15] (4);

   \draw [decorate,
    decoration = {calligraphic brace, mirror, raise=3pt}] (-2.7, 1.25) --  (-2.7,-1.25)
    node[pos=0.5,left=5pt,black]{ \footnotesize $m$ sources}; 
    
       \draw [decorate,
    decoration = {calligraphic brace, raise=3pt}] (-.5, .75) --  (-.5,-.75)
    node[pos=0.5,right=5pt,black]{ \footnotesize $\ell$ sinks};

     \draw[right hook->] (1.25,0) -- (2.25, 0); 

    \hspace{-1.25cm}

    \node[point,label={left:$ $}] (5) at (11.5,1) {};
   
   \node[point,label={left:$ $}] (7) at (11.5,-1) {};
   \node[point,label={right:$ $}] (8) at (10,0) {};
     \node[point,label={right:$ $}] (9) at (8.5,0) {};
        \node[point,label={left:$ $}] (10) at (6.5,.75) {};
   \node[point,label={left:$ $}] (12) at (6.5,-.75) {};

           \draw[decorate sep={.25mm}{2mm},fill] (11.5,.6)--(11.5,-.65);
         \draw[decorate sep={.25mm}{2mm},fill] (6.5,.5)--(6.5,-.65);

    \path[->]  
  (8) edge   [bend left=15] (5)
  (8) edge  [bend right=15]  (7)
  (9) edge  (8)
    (10) edge  [bend left=10]  (9)
  (12) edge [bend right=10]  (9); 
  
     \draw [decorate,
    decoration = {calligraphic brace, raise=15pt}] (11.5,1) --  (11.5,-1)
    node[pos=0.5,right=15pt,black]{\footnotesize $\ell$ sinks}; 
        \draw [decorate,
    decoration = {calligraphic brace, mirror, raise=10pt}] (6.5,.75) --  (6.5,-.75)
    node[pos=0.5,left=10pt,black]{\footnotesize $m+1$ sources}; 

\end{tikzpicture}  \] 
\end{enumerate}
\end{example}

\begin{rmk} \label{dim-wts-river-quiver-rmk} If $\gamma \in \ZZ^{(\Q_{\beta})_0}$ is an integral vector,  then $\I(\gamma) \in \ZZ^{\river_0}$ is the same as $\gamma$ at the vertices of the flags $\F(x_i)$, $i \in [m]$, and $\F(y_j)$, $j \in [\ell]$. Furthermore, we have that
$$
\I(\gamma)(x_0)=(n+1)C, \; \I(\gamma)(y_0)=nC, \text{~and~}\I(\gamma)(x_{m+1})=\ldots=\I(\gamma)(x_{m+n})=C,
$$
where $C:=\sum_{i=1}^m \gamma(x_i)$. Moreover, if $\widehat{\sigma}$ is a weight of $\river$ of the form $\langle \I(\gamma), \cdot \rangle_{\river}$, then 
\begin{enumerate}
\item $\widehat{\sigma}(x_0)=0$ and $\widehat{\sigma}(y_0)=-C$, and
\item $\widehat{\sigma}$ is equal to $\widetilde{\sigma}=\langle \gamma, \cdot \rangle_{\Q_\beta}$ at the vertices of the flags $\F(x_i)$, $i \in [m]$, and $\F(y_j)$, $j \in [\ell]$.
\end{enumerate}
\end{rmk}

\bigskip

\section{Hive-type polytopes for quiver multiplicities}\label{Hive-poly-semi-inv-sec}

\subsection{The irreducible representations of the general linear group} In this section we review the basics of the representation theory of the general linear group, which can be found in \cite{F3}. A \emph{partition} is a sequence $\lambda = (\lambda_1,\ldots, \lambda_r)$ of integers with $\lambda_1 \geq \ldots \geq \lambda_r \geq 0$. The \emph{length} of a partition, denoted by $\ell(\lambda)$, is defined to be the number of its nonzero parts. If $\lambda$ is a partition, we define $|\lambda|$ to be the sum of its parts. The Young diagram of a partition $\lambda$ is a collection of boxes, arranged in left-justified rows with $\lambda_i$ boxes in row $i$. If $a$ and $b$ are two positive integers, $(b^a)$ denotes the partition that has $a$ parts, all equal to $b$. We say that the diagram of $(b^a)$ is the $a \times b$ rectangle.

Now let $N$ be a fixed positive integer. Denote the set of partitions of length at most $N$ by $P_N$. For a partition $\lambda \in P_N$,  $S^\lambda V$ denotes the irreducible (polynomial) representation of $\GL(V)$ with highest weight $\lambda$, called a \emph{Schur module}, where $V$ is any fixed $N$-dimensional complex vector space.  Given partitions $\lambda, \mu, \nu \in P_N$, we define the \emph{Littlewood-Richardson coefficient} $c^\nu_{\lambda, \mu}$ to be the multiplicity of $S^\nu V$ in $S^\lambda V \otimes S^\mu V$, that is,
\[
c^\nu_{\lambda,\mu} = \dim_\CC \left( S^\nu V^* \otimes S^\lambda V \otimes S^\mu V\right)^{\GL(V)}.
\]
More generally, if $\nu, \lambda(1), \ldots, \lambda(r) \in P_N$, we define
\[
c^\nu_{\lambda(1),\ldots, \lambda(r)} = \dim_\CC \left(S^\nu V^* \otimes S^{\lambda(1)} V \otimes \cdots \otimes S^{\lambda(r)} V\right)^{\GL(V)}.
\]

\noindent
Following \cite{Zel-1999}, given partitions $\lambda(1), \ldots, \lambda(r) \in P_N$, we define partitions $\widetilde{\lambda}, \widetilde{\mu} \in P_{rN}$ by
\begin{equation}\label{Zel-trick-eqns}
\widetilde{\mu}_{(j-1)N+i}:=\sum_{k=j+1}^r \lambda_1(k) \text{~and~} \widetilde{\lambda}_{(j-1)N+i}=\lambda_i(j)+\widetilde{\mu}_{(j-1)N+i}, \forall j \in [r], i \in [N].
\end{equation}

\begin{rmk} 
\begin{enumerate}
\item The last $N$ parts of the partition $\widetilde{\mu}$ are zero. Furthermore, $\widetilde{\lambda}- \widetilde{\mu}$ is a skew diagram whose connected components are translates of the diagrams of $\lambda(1),\ldots, \lambda(r)$.

\item We emphasize that if the partitions $\lambda(1), \ldots, \lambda(r)$ have different lengths, we first choose an integer $N \geq 1$ such that $\ell(\lambda(1)), \ldots, \ell(\lambda(r)) \leq N$ and extend each $\lambda(i)$ by adding $N-\ell(\lambda(i))$ zero parts. Then we construct the partitions $\widetilde{\lambda}$ and $\widetilde{\mu}$ according to Equation $(\ref{Zel-trick-eqns})$. This is emphasized in the diagram below by using red vertical lines to indicate that zeros may have been added to the end of the partitions.
\end{enumerate}
\end{rmk}

\noindent
Diagrammatically, these partitions are defined as 

\[
\begin{tikzpicture}[scale=.4]
\coordinate (a) at (0,0);

\coordinate (lr_1) at (1,-12.25);
\coordinate (lr_2) at (1,-13);
\coordinate (lr_2b) at (0,-13);
\coordinate (lr_3) at (0,-10);
\coordinate (lr_4) at (4,-10);
\coordinate (lr_5) at (4,-10.75);
\coordinate (lr_6) at (3,-10.75);
\coordinate (lr_7) at (3,-11.25);
\draw[black,thick] (lr_1)--(lr_2)--(lr_2b)--(lr_3)--(lr_4)--(lr_5)--(lr_6)--(lr_6)--(lr_7);
\coordinate (lr_v) at (0,-14);
\draw[red,thick] (lr_2b) -- (lr_v);

\coordinate (l2_1) at (7.5,-6.5);
\coordinate (l2_1b) at (7.5,-7);
\coordinate (l2_2) at (7,-7);
\coordinate (l2_3) at (7,-4);
\coordinate (l2_4) at (10,-4);
\coordinate (l2_5) at (10,-5);
\coordinate (l2_6) at (8.5,-5);
\coordinate (l2_7) at (8.5,-5.5);
\draw[black,thick] (l2_1)--(l2_1b)--(l2_2)--(l2_3)--(l2_4)--(l2_5)--(l2_6)--(l2_6)--(l2_7);
\coordinate (l2_v) at (7,-8);
\draw[red,thick] (l2_v) -- (l2_2);

\coordinate (l1_1) at (10.75,-1.75);
\coordinate (l1_2) at (10.75,-2.5);
\coordinate (l1_3) at (10,-2.5);
\coordinate (l1_4) at (10,0);
\coordinate (l1_5) at (13,0);
\coordinate (l1_6) at (13,-.75);
\coordinate (l1_7) at (12.5,-.75);
\coordinate (l1_8) at (12.5,-1.25);
\draw[black,thick] (l1_1)--(l1_2)--(l1_3)--(l1_4)--(l1_5)--(l1_6)--(l1_6)--(l1_7)--(l1_8);
\coordinate (l1_v) at (10,-4);
\draw[red,thick] (l1_v)--(l1_3);

\draw[black,thick] (a)--(l1_5);
\draw[black,thick] (a)--(lr_2b);

\node (l) at (-1,-5) {$\widetilde{\lambda}=~$};
\node (mu) at (3,-3) {$\widetilde{\mu}$};
\node (lr) at (1.5,-11) {$\lambda(r)$};
\node (l2) at (8,-4.5) {$\lambda(2)$};
\node (l1) at (11.3,-.7) {$\lambda(1)$};

\draw[decorate sep={.5mm}{2mm},fill] (1.2,-12.1) -- (2.9,-11.3);
\draw[decorate sep={.5mm}{2mm},fill] (4.2,-9.8) -- (6.8,-8);
\draw[decorate sep={.5mm}{1.8mm},fill] (7.7,-6.3) -- (8.3,-5.5);
\draw[decorate sep={.5mm}{2mm},fill] (11.1,-1.75) -- (12.1,-1.3);
\end{tikzpicture}
\]

\bigskip
\begin{prop}\cite[Proposition 9]{Zel-1999}\label{Zelevinsky-trick} Keep the same notations as above. If $\nu, \lambda(1),\ldots, \lambda(r) \in P_N$ are partitions, then 
$$c^\nu_{\lambda(1),\ldots, \lambda(r)} = c^{\widetilde{\lambda}}_{\widetilde{\mu},\nu}.$$
\end{prop}

We end this subsection by listing some very useful properties of the irreducible representations of $\GL(V)$. 

\begin{prop}\phantomsection\label{irr-gln-prop}
\begin{enumerate}[\normalfont(1)]
\item \hypertarget{irr-gln-prop-1}{} Let $\lambda \in P_N$. Then $\left(S^\lambda(V)\right)^{\SL(V)} \neq 0$ if and only if $\dim S^\lambda (V) = 1$ if and only if $\lambda = (w^N)$. In this case, $\left(S^\lambda V\right)^{\SL(V)}$ is spanned by one semi-invariant of weight $w$.
\item  \hypertarget{irr-gln-prop-2}{}  Let $\lambda = (\lambda_1,\ldots, \lambda_N)$ and $\mu = (\mu_1,\ldots, \mu_N)$ be two partitions. Then $\left(S^\lambda V^* \otimes S^\mu V\right)^{\SL(V)} \neq 0$ if and only if $\mu_i-\lambda_i=w$ for all $i \in [N]$ for some integer $w$. If this is the case, $\left(S^\lambda V^* \otimes S^\mu V\right)^{\SL(V)}$ is a one-dimensional vector space spanned by a semi-invariant of weight $w$.
\item \hypertarget{irr-gln-prop-3}{}  Let $U$ be a rational representation of $\GL(V)$. Then $U^{\SL(V)} = \bigoplus_{\theta \in \ZZ} U_\theta$, where
\[
U_\theta = \{u \in U \mid g \cdot u = \det(g)^\theta \cdot u, \; g \in \GL(V)\}
\]
is the space of semi-invariants of weight $\theta$. Moreover, $U_\theta = \left(U \otimes \det^{-\theta}_V\right)^{\GL(V)}$, where $\det^{-\theta}_V : \GL(V) \to \CC^*$ is the one-dimensional representation of $\GL(V)$ that sends $g \in \GL(V)$ to $\det^{-\theta}(g) \in \CC^*$.
\end{enumerate}
\end{prop}
	
\subsection{Knutson-Tao's hive polytopes for Littlewood-Richardson coefficients}
\label{subsection:hives}
In this subsection we review a combinatorial model for computing Littlewood-Richardson coefficients that was introduced by A. Knutson and T. Tao in \cite{KT} and \cite{KT2}. Further details about this combinatorial description and its consequences can be found in, for instance, \cite{KTT}, \cite{KTT2}, \cite{KTT3}, and \cite{KTT4}.

To define the polytope whose number of lattice points is the Littlewood-Richardson coefficient $c^\nu_{\lambda,\mu}$ for a specific choice of partitions $\nu,\lambda,$ and $\mu$ with at most $N$ parts, we start by considering a triangular graph obtained by dividing an equilateral triangle into $N^2$ smaller equilateral triangles of the same size by plotting $N+1$ vertices along each edge of the large triangle.  

An \emph{$N$-hive} is a tuple of numbers $(e_{i,j},  f_{i,j}, g_{i,j})$ with $0 \leq i,j, i+j \leq N-1$ where the entries $e_{i,j}$ label  the edges parallel to the left boundary of the large triangle, the entries $f_{i,j}$ label the edges parallel to the right boundary of the large triangle, and the entries $g_{i,j}$ label the horizontal edges. Furthermore, these numbers must satisfy the hive conditions $(\ref{hive-ineq-1})-(\ref{hive-ineq-2})$ described below. A hive is said to be an \emph{integral hive} if all of its entries are non-negative integers. A $3$-hive is depicted in Figure \hyperlink{LR_hive}{1} below.

The hive conditions are a set of constraints on the edge labels of each the following two elementary triangles and three elementary rhombi: 

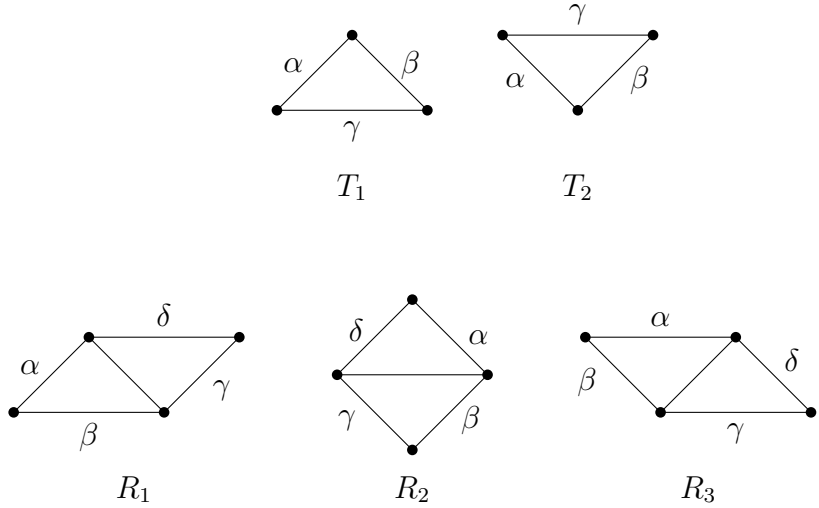
\[
\begin{tikzpicture}

\coordinate [label={left: $\alpha$}] () at (-9.5,0.1);
\coordinate [label={right: $\beta$}] () at (-8.5,0.1);
\coordinate [label={below: $\gamma$}] () at (-9,-.5);
\coordinate [label={below left:}] (a) at (-10,-.5);
\coordinate [label={above left:}] (b) at (-8,-.5);
\coordinate [label={below right:}] (c) at (-9,.5);
\draw (a)--(b)--(c)--cycle;

\coordinate [label={below: $T_1$}] () at (-9,-1.2);

\foreach \i in {a,b,c}
  \fill (\i) circle (2pt);

\coordinate [label={left: $\alpha$}] () at (-6.55,-.1);
\coordinate [label={right: $\beta$}] () at (-5.45,-.1);
\coordinate [label={above: $\gamma$}] () at (-6,.5);
\coordinate [label={below left:}] (a) at (-6,-.5);
\coordinate [label={above left:}] (b) at (-7,.5);
\coordinate [label={below right:}] (c) at (-5,.5);
\draw (a)--(b)--(c)--cycle;

\coordinate [label={below: $T_2$}] () at (-6,-1.2);

\foreach \i in {a,b,c}
  \fill (\i) circle (2pt);
\end{tikzpicture}
\]

\smallskip

\[
\begin{tikzpicture}

\coordinate [label={above: $\delta$}] () at (-2,.5);
\coordinate [label={below: $\beta$}] () at (-3,-.5);
\coordinate [label={left: $\alpha$}] () at (-3.5,0.1);
\coordinate [label={right: $\gamma$}] () at (-1.5,-.2);
\coordinate [label={below left:}] (a) at (-4,-.5);
\coordinate [label={above left:}] (b) at (-3,.5);
\coordinate [label={below right:}] (c) at (-2,-.5);
\coordinate [label={above:}] (d) at (-1,.5);
\draw (a)--(b)--(d)--(c)--cycle;
\draw (b)--(c);

\coordinate [label={below: $R_1$}] () at (-2.4,-1.2);

\foreach \i in {a,b,c,d}
  \fill (\i) circle (2pt);

  \hspace{-.7cm}
  
\coordinate [label={left: $\delta$}] () at (1.5,.6);
\coordinate [label={right: $\alpha$}] () at (2.6,.5);
\coordinate [label={left: $\gamma$}] () at (1.4,-.6);
\coordinate [label={right: $\beta$}] () at (2.5,-.6);
\coordinate [label={above:}] (a) at (2,1);
\coordinate [label={left:}] (b) at (1,0);
\coordinate [label={right:}] (c) at (3,0);
\coordinate [label={below:}] (d) at (2,-1);
\draw (a)--(b)--(d)--(c)--cycle;
\draw (b)--(c);

\foreach \i in {a,b,c,d}
  \fill (\i) circle (2pt);

\coordinate [label={below: $R_2$}] () at (2,-1.2);

 \hspace{-.7cm}
\coordinate [label={left: $\beta$}] () at (5.3,-.1);
\coordinate [label={above: $\alpha$}] () at (6,.5);
\coordinate [label={below: $\gamma$}] () at (7,-.5);
\coordinate [label={right: $\delta$}] () at (7.5,0.2); 
\coordinate [label={above:}] (a) at (5,.5);
\coordinate [label={above right:}] (b) at (7,.5);
\coordinate [label={below: }] (c) at (6,-.5);
\coordinate [label={below right:}] (d) at (8,-.5);
\draw (a)--(b)--(d)--(c)--cycle;
\draw (b)--(c);

\foreach \i in {a,b,c,d}
  \fill (\i) circle (2pt);

  \coordinate [label={below: $R_3$}] () at (6.5,-1.2);
\end{tikzpicture}
\]

In each of the two triangles $T_1$ and $T_2$, we want 
\begin{equation}\label{hive-ineq-1}
\alpha + \beta = \gamma.
\end{equation}
In particular, this implies that in the three rhombi with our labeling, we must have 
\begin{equation}
\alpha + \delta = \beta + \gamma.
\end{equation}
Furthermore, we want the elementary rhombi to satisfy the \emph{rhombus inequalities}, \emph{i.e.},  for each of $R_1, R_2,$ and $R_3$, we want
\begin{equation}\label{hive-ineq-2}
\alpha \geq \gamma \quad \text{ and } \quad \beta \geq \delta,
\end{equation}
where it is clear that either one of the two inequalities in $(\ref{hive-ineq-2})$ implies the other one. Moreover, note that inequalities $(\ref{hive-ineq-1})-(\ref{hive-ineq-2})$ define a convex polyhedral cone in $\RR^{\frac{3N(N+1)}{2}}$. 

\begin{definition}
An \emph{LR-hive} is an integer $N$-hive whose border labels are determined by three partitions $\lambda$, $\mu$, and $\nu$ with at most $N$ non-zero parts such that $|\nu|=|\lambda|+|\mu|$ and
$$
e_{i,0}=\lambda_{i+1}, \;\; f_{j, N-1-j}=\mu_{N-j}, \text{~and~}g_{0,k}=\nu_{k+1}, \;\;  \forall \, 0 \leq i,j,k \leq N-1. 
$$
\end{definition}
\begin{figure}[tbp] 
\hypertarget{LR_hive}{}
\begin{center}
\includegraphics[trim=2cm 18cm 2cm 2cm,clip,scale=.9]{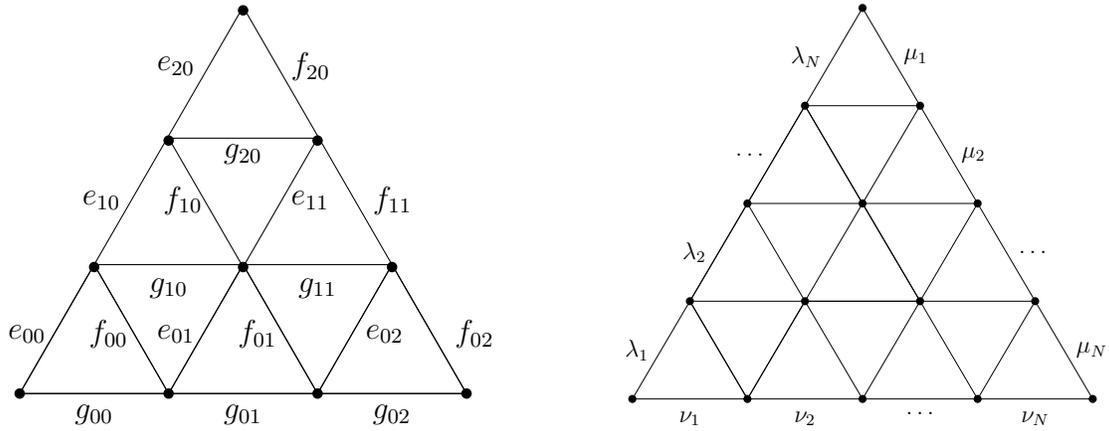}
\caption{Left: The $3$-hive with border labels. Right: Boundary labels determined by partitions $\lambda, \mu, \nu$.}
\end{center}
\end{figure} 

\begin{theorem}[\cite{KT}, Theorem 4]\label{LR-hives-KnuTao-thm}
Let $\lambda, \mu$, and $\nu$ be three partitions with at most $N$ nonzero parts such that $|\nu|=|\lambda|+|\mu|$.  Then the Littlewood-Richardson coefficient $c^\nu_{\lambda, \, \mu}$ is the number of LR-hives with boundary labels determined by $\lambda, \mu$, and $\nu$.
\end{theorem}

\subsection{Computing weight spaces of semi-invariants via Littlewood-Richardson coefficients} \label{compute-semi-inv-river-sec}
In this section we compute weight spaces of semi-invariants for the quiver $\river$. The computational method we use in this paper has been pioneered by Derksen and Weyman (see for example \cite{DW1, DW2}) who used it to great effect to prove the Saturation Conjecture for Littlewood-Richardson coefficients as a consequence of their more general Saturation Property for quiver semi-invariants. 

Let $m,n$, and $\ell$ be positive integers and let $\Q$ be the $n$-complete bipartite quiver with source vertices $x_1, \ldots, x_m$, and sink vertices $y_1, \ldots, y_{\ell}$. Let $\beta$ be a sincere dimension vector of $\Q$.

Let $\river$ be the quiver introduced in Section \ref{qes-embedding-sec}. Our goal in this section is to find a hive-type polytopal description for the weight spaces of semi-invariants for the quiver set-up $(\river, \widehat{\beta})$, where $\widehat{\beta}=\I(\widetilde{\beta})$ with $\beta$ being a sincere dimension vector of $\Q$ and $\widetilde{\beta}$  its extension to $\Q_{\beta}$. More precisely, we have that $\widehat{\beta}$ is the dimension vector of $\river$ given by:
\begin{itemize}
\item $\widehat{\beta}(x_0)=(n+1)d$, $\widehat{\beta}(y_0)=nd$, and $\widehat{\beta}(x_{m+k})=d$ for all $k \in [n]$, where 
$$
d:=\sum_{i=1}^m \beta(x_i);
$$

\item traversing the flag $\F(x_i)$ going into the vertex $x_i$ from left to right, the values of $\widehat{\beta}$ at the vertices of this flag  are $1, 2, \ldots, \beta(x_i)$ for every $i \in [m]$;\\

 \item traversing the flag $\F(y_j)$ going out of the vertex $y_j$ from left to right, the values of $\widehat{\beta}$ at the vertices of this flag  are $\beta(y_j), \ldots, 2, 1$ for every $j \in [\ell]$.
\end{itemize} 

Next, let $\widehat{\sigma}$ be a weight of $\river$ such that $\widehat{\sigma} \cdot \widehat{\beta}=0$. Furthermore, we assume that:
\begin{equation}\label{eqn1-sigma-hat}
\restr{\widehat{\sigma}}{\F(x_i)} \geq 0, \, i \in [m], \text{~and~} \restr{\widehat{\sigma}}{\F(y_j)} \leq 0, \, j \in [\ell],
\end{equation}
and
\begin{equation}\label{eqn2-sigma-hat}
\widehat{\sigma}(x_0)=0, \, \text{~and~} \widehat{\sigma}(x_{m+k})=-\widehat{\sigma}(y_0) \geq 0, \, \forall k \in [n].
\end{equation}

\noindent
For each $i \in [m]$, let us label the vertices of the flag $\F(x_i)$  of the quiver $\river$ as follows
$$ 
\F(x_i) : \begin{array}{c}\\{\bullet} \\ i_1 \end{array}  \rightarrow \begin{array}{c}\\{\bullet} \\ i_2 \end{array} \rightarrow \cdots \begin{array}{c}\\{\bullet} \\ i_{\beta(x_i)-1} \end{array} \longrightarrow \hspace{-6pt} \begin{array}{c} x_i\\ \bullet \\i_{\beta(x_i)} \end{array}
$$
and define the partition 
\begin{equation}\label{defn-lambda-sigma-hat-eqn1}
\lambda(i)=\left( \sum_{k \leq r \leq \beta(x_i)} \widehat{\sigma}(i_r) \right)_{k \in [\beta(x_i)]} \in P_{\beta(x_i)}
\end{equation}

For each $j \in [\ell]$, let us label the vertices of the flag $\F(y_j)$ of the quiver $\river$ as follows
$$
\F(y_j) : \begin{array}{c}y_j\\ \bullet \\ j_{\beta(y_j)} \end{array} \hspace{-6pt} \longrightarrow \begin{array}{c}\\{\bullet} \\ j_{\beta(y_j)-1} \end{array} \longrightarrow \cdots \longrightarrow \begin{array}{c}\\{\bullet} \\ j_2 \end{array}  \longrightarrow \begin{array}{c}\\{\bullet} \\ j_1 \end{array}
$$
and define the partition 
\begin{equation}\label{defn-nu-sigma-hat-eqn2}
\nu(j)=\left( -\sum_{k \leq r \leq \beta(y_j)} \widehat{\sigma}(j_r) \right)_{k \in [\beta(y_j)]} \in P_{\beta(y_j)}
\end{equation}

We point out that since $\widehat{\sigma}\cdot \widehat{\beta}=0$ we have that
$$
\sum_{i=1}^m |\lambda(i)|=\sum_{j=1}^{\ell} |\nu(j)|.
$$

\begin{prop} \label{main-compute-semi-inv-LR-prop} Let $\widehat{\sigma}$ be a weight of $\river$ with $\widehat{\sigma} \cdot \widehat{\beta}=0$, and such that $\widehat{\sigma}$ satisfies $(\ref{eqn1-sigma-hat})$ and $(\ref{eqn2-sigma-hat})$.  

Then the following formula holds:
\begin{equation}\label{main-compute-semi-inv-LR-eqn}
\dim \SI(\river, \widehat{\beta})_{\widehat{\sigma}} = \sum_{\substack{\mu\\ \ell(\mu)\leq nd}} c^\mu_{\lambda(1),\ldots, \lambda(m), \underbrace{(f^d),\ldots, (f^d)}_{n \text{ times }}} \cdot c^\mu_{\nu(1),\ldots, \nu(\ell),(f^{nd})},
\end{equation} 
where $f=-\widehat{\sigma}(y_0)$.
\end{prop}

\begin{proof} To find the desired formula for $\dim \SI(\river, \widehat{\beta})_{\widehat{\sigma}}$, we proceed as follows. First, we use Cauchy's formula to decompose $ \dsp \CC[\rep(\river, \widehat{\beta})]$ into a direct sum of irreducible representations of $\GL(\widehat{\beta})$. Then, we consider the ring of semi-invariants $\dsp \SI(\river, \widehat{\beta}) = \CC[\rep(\river,\widehat{\beta})]^{\SL(\hb)}$ and sort out those semi-invariants that have weight $\widehat{\sigma}$.

For each $i \in [m]$, let us focus on the following subquiver of $\river$:

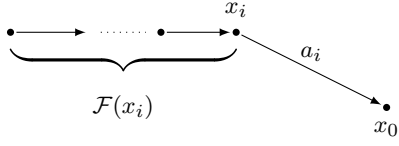
\begingroup\makeatletter\def\f@size{9.5}\check@mathfonts
\[
\vcenter{\hbox{  
\begin{tikzpicture}[point/.style={shape=circle, fill=black, scale=.3pt,outer sep=3pt},>=latex, decoration=snake]
   \node[point,label={below:$x_0$}] (1) at (-1.5,0) {};
        \node[point,label={above:$x_{i}$}](3) at (-3.5,1) {};
           \node[point,label={}](4) at (-4.5,1) {};
              \node[point,label={}](6) at (-6.5,1) {};
  
        \draw[dotted] (-5.3,1)--(-4.7, 1);
       
    \path[->]  
    (4) edge (3)
    (6) edge (-5.5,1)
    (3) edge node[above] {$a_i$}  (1);  
    
 \draw[B] (-6.5,1.3) -- node[below=10mm] {$\F(x_i)$} (-3.5,1.3);
\end{tikzpicture} 
}}\]
\endgroup

For convenience, let us denote $\beta(x_i) = r$ and write $ V_k = \CC^k$, $\forall \, 1 \leq k \leq r$, and $V = \CC^{\hb(x_0)} = \CC^{(n+1)d}$. Then the contribution of the subquiver above to $\CC[\rep(\river, \hb)]$ is: 
\begin{align*}
& \CC\left[\prod_{k=1}^{r-1} \Hom(V_k, V_{k+1}) \times \Hom(V_r, V) \right] \\
& = \bigotimes_{k=1}^{r-1} S\left(V_k \otimes V_{k+1}^* \right) \otimes S(V_r \otimes V^*) \\ 
& = \bigoplus_{\gamma(1),\ldots, \gamma(r-1), \gamma(i)} S^{\gamma(1)}(V_1) \otimes \bigotimes_{k=2}^{r-1} \left( S^{\gamma(k-1)} V_k^* \otimes S^{\gamma(k)} V_k \right ) \otimes \left ( S^{\gamma(r-1)} V_r^* \otimes S^{\gamma(i)} V_r \right) \otimes S^{\gamma(i)} V^* 
\end{align*} 

This yields the following contribution of the vertices of the flag $\F(x_i)$ to $\SI(\river, \hb)$: 
\[ \bigoplus_{\gamma(1),\ldots, \gamma(r-1), \gamma(i)} \left (S^{\gamma(1)} V_1 \right)^{\SL(V_1)} \otimes \bigotimes_{k=2}^{r-1} \left ( S^{\gamma(k-1)} V_k^* \otimes S^{\gamma(k)} V_k \right)^{\SL(V_k)} \otimes \left ( S^{\gamma(r-1)} V_r^* \otimes S^{\gamma(i)} V_r \right)^{\SL(V_r)} \]

Sorting out those semi-invariants of weight $\widehat{\sigma}$ completely determines the partitions $\gamma(1), \ldots, \gamma(r-1)$, and $\gamma(i)$. By Proposition \hyperlink{irr-gln-prop-1}{\ref{irr-gln-prop}(1)}, we have that $\dsp \left ( S^{\gamma(1)}  V_1 \right)^{\SL(V_1)} \neq 0$ if and only if it is one-dimensional. If this is the case, then $\gamma(1)$ is a $1 \times w$ rectangle with $w \in \NN$ and $\dsp \left ( S^{\gamma(1)} V_1 \right)^{\SL(V_1)}$ is spanned by a semi-invariant of weight $w$. Thus, $\left(S^{\gamma(1)} V_1\right)^{\SL(V_1)}$ contains a semi-invariant of weight $\wsi{i_1}$ if and only if $\gamma(1) = (\wsi{i_1})$. 

Next, using Proposition \hyperlink{irr-gln-prop-2}{\ref{irr-gln-prop}(2)}, we have that the space $\left(S^{\gamma(1)}V_2^* \otimes S^{\gamma(2)}_2\right)^{\SL(V_2)}$ is nonzero if and only if it is one-dimensional. If that is the case, then $\gamma(2)$ is $\gamma(1)$ plus some extra columns of height $2$, with the number of these extra columns equaling the weight of the semi-invariant spanning $\left(S^{\gamma(1)}V_2^* \otimes S^{\gamma(2)}_2\right)^{\SL(V_2)}$. Thus, this space contains a nonzero semi-invariant of weight $\wsi{i_2}$ if and only if $\gamma(2) = (\wsi{i_2} + \wsi{i_1}, \wsi{i_2})$. Continuing with this reasoning, we see that $\gamma(1), \ldots, \gamma(r-1)$, and $\gamma(i)$ are completely determined by $\widehat{\sigma}$ with 
\[
\gamma(i) = (\wsi{x_i} + \wsi{i_{r-1}}+\ldots + \wsi{i_1}, \ldots, \wsi{x_i}),
\]
which is precisely $\lambda(i)$. Now, let us focus on vertex $x_0$ and its neighbors:

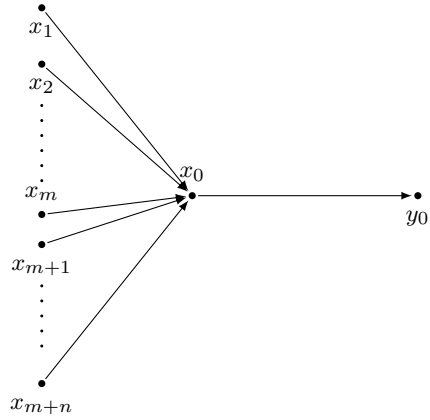
\begingroup\makeatletter\def\f@size{9.5}\check@mathfonts
$$
\vcenter{\hbox{  
\begin{tikzpicture}[point/.style={shape=circle, fill=black, scale=.3pt,outer sep=3pt},>=latex, decoration=snake]
   \node[point,label={below:$y_0$}] (0) at (1.5,0) {};
   \node[point,label={above:$x_0$}] (1) at (-1.5,0) {};

       \node[point,label={below:$x_{1}$}](2)  at (-3.5,2.5) {};
        \node[point,label={below:$x_{2}$}](3) at (-3.5,1.75) {};
      \node[point,label={above:$x_{m}$}](4) at (-3.5, -.25) {};
       \node[point,label={below:$x_{m+1}$}](5) at (-3.5, -.65) {};
     \node[point,label={below:$x_{m+n}$}](6) at (-3.5, - 2.5){};

    \path[->]  

  (1) edge node[above] {} (0)  
   (2) edge  node[above] {}  (1) 
    (3) edge node[above] {}  (1) 
     (4) edge node[above] {}   (1) 
      (5) edge node[below] {}  (1) 
      (6) edge  node[below, xshift=1em] {} (1);
   
         \draw[decorate sep={.25mm}{2mm},fill] (-3.5, 1.2)--(-3.5,.2);     
         \draw[decorate sep={.25mm}{2mm},fill] (-3.5, -1.2)--(-3.5,-2.1); 

\end{tikzpicture} 
}}
$$
\endgroup

We write $W = \CC^{\widehat{\beta}(y_0)} = \CC^{nd}$. The contribution of this subquiver to $\CC[\rep(\river,\widehat{\beta})]$ is 

\[
\CC[\Hom(V_1,V) \times \cdots \times \Hom(V_{m+n},V) \times \Hom(V,W)] = S(V_1\otimes V^*) \otimes \cdots \otimes S(V_{m+n} \otimes V^*) \otimes S(V\otimes W^*).
\]

Using Cauchy's Formula again, we can write
\begin{equation}
S(V\otimes W^*)=\bigoplus S^{\mu}(V)\otimes S^{\mu}(W^*),
\end{equation}
where the sum is over all partitions $\mu$ of length at most $\min \{\dim V, \dim W \}=nd$. Since the weight $\widehat{\sigma}$ is zero at vertex $x_0$, the calculations above together with Proposition \hyperlink{irr-gln-prop-3}{\ref{irr-gln-prop}(3)} show that the contribution of $x_0$ to $\SI(\river,\widehat{\beta})_{\widehat{\sigma}}$ is made of spaces of the form

\[
\left(S^{\lambda(1)} V^* \otimes \cdots \otimes S^{\lambda(m)} V^* \otimes \underbrace{S^{(f^d)} V^* \otimes \cdots \otimes S^{(f^d)} V^*}_{n \text{ times}} \otimes S^\mu V\right)^{\GL(V)},
\]
with $\mu$ a partition of length at most $nd$.

Taking into account the contributions of all the other vertices of $\river$, we get that $\SI(\river, \widehat{\beta})_{\widehat{\sigma}}$ is isomorphic to 

\begin{multline*}
\bigoplus_{\substack{\mu \\ \ell(\mu) \leq nd}} \left(S^{\lambda(1)} V^* \otimes \cdots \otimes S^{\lambda(m)} V^* \otimes \underbrace{S^{(f^d)} V^* \otimes \cdots \otimes S^{(f^d)} V^*}_{n \text{ times}} \otimes S^\mu V\right)^{\GL(V)}\otimes \\
\otimes \left(S^{\nu(1)} W \otimes \cdots S^{\nu(\ell)} W \otimes S^{\mu} W^* \otimes \mathrm{det}^f_W\right)^{\GL(W)}.
\end{multline*}
Thus, we conclude that 
\[
\dim \SI(\river, \widehat{\beta})_{\widehat{\sigma}} = \sum_{\substack{\mu\\ \ell(\mu)\leq nd}} c^\mu_{\lambda(1),\ldots, \lambda(m), \underbrace{(f^d),\ldots, (f^d)}_{n \text{ times }}} \cdot c^\mu_{\nu(1),\ldots, \nu(\ell),(f^{nd})}.
\]
\end{proof}

\begin{remark}\label{main-compute-one-sink-rmk} We point out that when $\ell=1$, \emph{i.e.}, $\Q$ has only one sink vertex and thus $\river$ is a star quiver, the right hand side of $(\ref{main-compute-semi-inv-LR-eqn})$ can be simplified down to one multiple Littlewood-Richardson coefficient. Indeed, for a partition $\mu$ with $\ell(\mu) \leq nd$, we have $c^{\mu}_{\nu(1), (f^{nd})}\neq 0$ if and only if $(S^{\mu}(W)^*\otimes S^{\nu(1)}(W)\otimes \det_W^f)^{\GL(W)}\neq 0$, where $W=\CC^{nd}$.  By Proposition \hyperlink{irr-gln-prop-3}{\ref{irr-gln-prop}(3)}, this is further equivalent to saying that the weight space of weight $-f$ that occurs in the weight space decomposition of $(S^{\mu}(W)^*\otimes S^{\nu(1)}(W))^{\SL(W)}$ is not zero. Finally, using Proposition \hyperlink{irr-gln-prop-3}{\ref{irr-gln-prop}(2)} , we see that this is equivalent to $\mu$ being equal to $\nu(1)$ plus $f$ columns of length $nd$ and $c^{\mu}_{\nu(1), (f^{nd})}=1$.  Thus, we get that
$$
\dim \SI(\river, \widehat{\beta})_{\widehat{\sigma}}=c^{\nu(1)+(f^{nd})}_{\lambda(1), \ldots, \lambda(m), \underbrace{ (f^d),\ldots, (f^d)}_{n \text{ times }}}.
$$
This can be further expressed as a single Littlewood-Richardson coefficient via Proposition \ref{Zelevinsky-trick}. 
\end{remark}

With Proposition \ref{main-compute-semi-inv-LR-prop} at our disposal, we are ready to establish the following formula for the multiplicities $K^{\beta}_{\unlm}$. 

\begin{theorem} \label{main-formula-mult-thm} Let $\Q$ be an $n$-complete bipartite quiver with source vertices $x_1, \ldots, x_m$ and sink vertices $y_1, \ldots, y_\ell$ and let $\beta=(\beta(x))_{x \in Q_0}$ be a sincere dimension vector of $\Q$. 

Let $\unlm=(\lm(x_i), -\lm(y_j))_{i \in [m], j \in [\ell] }$ be a tuple of sequences with $\lm(x_i)$ a partition of length at most $\beta(x_i)$ and $\lm(y_j)$ a partition of length at most $\beta(y_j)$ such that
$$
\sum_{i=1}^m |\lambda(x_i)|=\sum_{j=1}^{\ell} |\lambda(y_j)|.
$$

Then
\begin{equation}\label{main-formula-mult}
K^{\beta}_{\unlm}=\sum_{\substack{\mu\\ \ell(\mu)\leq nd}} c^\mu_{\lambda(x_1),\ldots, \lambda(x_m), \underbrace{(f^d),\ldots, (f^d)}_{n \text{ times }}} \cdot c^\mu_{\lm(y_1),\ldots, \lm(y_\ell),(f^{nd})},
\end{equation}
where 
$$
d=\sum_{i \in [m]}\beta(x_i) \text{~and~} f=\sum_{i \in [m]} \lambda_1(x_i), \text{~the sum of the largest parts of the partitions} \lm(x_i).
$$
\end{theorem}

\begin{proof} From the tuple $\unlm=(\lm(x_i), -\lm(y_j))_{i \in [m], j \in [\ell]}$, we can construct the following weight $\widetilde{\sigma}_{\unlm}$ of $\Q_\beta$.   If $x$ is a source vertex of $\Q$, the values of $\widetilde{\sigma}_{\unlm}$ along the $\beta(x)$ vertices of the flag 
\[
\F(x) :  \bullet \rightarrow \bullet \rightarrow \cdots \bullet \rightarrow \hspace{-6pt} \begin{array}{c}\\ \bullet \\ x \end{array}
\]
are 
\begin{equation} \label{wt-part-eqn1}
\lambda_1(x) - \lambda_2(x), \, \ldots, \, \lambda_{\beta(x)-1}(x) - \lambda_{\beta(x)}(x), \, \lambda_{\beta(x)}(x).
\end{equation}

If, instead, $y$ is a sink vertex of $\Q$, the values of $\widetilde{\sigma}_{\unlm}$ along the $\beta(y)$ vertices of the flag
\[
\F(x) : \begin{array}{c}\\ \bullet \\ y \end{array} \hspace{-6pt} \rightarrow \bullet \rightarrow \cdots \bullet \rightarrow \bullet
\]
are 
\begin{equation}\label{wt-part-eqn2}
-\lambda_{\beta(y)}(y), \,  \lambda_{\beta(y)}(y) - \lambda_{\beta(y)-1}(y), \, \ldots, \, \lambda_2(y) - \lambda_1(y).
\end{equation}

Using the same methodology as in the proof of Proposition \ref{main-compute-semi-inv-LR-prop}, we can express both $\dim \SI(\Q_\beta, \widetilde{\beta})_{\widetilde{\sigma}_{\unlm}}$ and $K^{\beta}_{\unlm}$ in terms of sums of products of multiple Littelwood-Richardson coefficients. Specifically, we obtain that  
\begin{equation}\label{eqn-mult-wt-semi-inv}
\dim \SI(\Q_\beta, \widetilde{\beta})_{\widetilde{\sigma}_{\unlm}}=K^{\beta}_{\unlm}=\sum_{\mu_{i,j}^{(r)}}\ \  \prod_{i=1}^m c^{\lambda(x_i)}_{\mu_{i,1}^{(1)}, \ldots, \mu_{i,1}^{(n)}, \ldots, \mu_{i,\ell}^{(1)}, \ldots, \mu_{i,\ell}^{(n)}} \cdot \prod_{j=1}^{\ell} c^{\lambda(y_j)}_{\mu_{1,j}^{(1)}, \ldots, \mu_{1,j}^{(n)}, \ldots, \mu_{m,j}^{(1)}, \ldots, \mu_{m,j}^{(n)}},
\end{equation}
where the sum is over all partitions $\mu^{(r)}_{i,j}$, $i \in [m]$, $j \in [\ell]$, $r \in [n]$, with $\ell(\mu^{(r)}_{i,j}) \leq \min\{\beta(x_i), \beta(y_j)\}$.

\smallskip
Our goal is to simplify this complex formula. We start by expressing $\widetilde{\sigma}_{\unlm}$ as $\langle \alpha, \cdot \rangle_{\Q_\beta}$, and then consider the weight $\widehat{\sigma}:=\langle \I(\alpha), \cdot \rangle_{\river}$ of $\river$.  By construction, we have that $\restr{\widetilde{\sigma}_{\unlm}}{\F(x_i)} \geq 0, \, \forall i \in [m]$, and $\restr{\widetilde{\sigma}_{\unlm}}{\F(y_j)} \leq 0, \, \forall j \in [\ell]$. Also, it is immediate to see $\alpha(x_i)=\lm_1(x_i)$, $\forall i \in [m]$, and so 
$$
\widehat{\sigma}(x_{m+1})=\ldots=\widehat{\sigma}(x_{m+1}) =-\widehat{\sigma}(y_0)=\sum_{i \in [m]} \lambda_1(i) \geq 0,
$$
by Remark \ref{dim-wts-river-quiver-rmk}. Thus, the weight $\widehat{\sigma}$ satisfies $(\ref{eqn1-sigma-hat})$ and $(\ref{eqn2-sigma-hat})$. Furthermore, it follows from $(\ref{wt-part-eqn1})$ and $(\ref{wt-part-eqn2})$ that $\lambda(x_i)$ and $\lambda(y_j)$ are precisely the partitions $\lambda(i)$ and $\nu(j)$ from $(\ref{defn-lambda-sigma-hat-eqn1})$ and $(\ref{defn-nu-sigma-hat-eqn2})$, respectively.

Next we claim that
\begin{equation}\label{eqn-embed-1}
\dim \SI(\Q_\beta, \widetilde{\beta})_{\widetilde{\sigma}_{\unlm}}=\dim \SI(\river, \widehat{\beta})_{\widehat{\sigma}}.
\end{equation}
Indeed, we can see via Remark \ref{dim-wts-river-quiver-rmk} that $\alpha$ is a dimension vector of $\Q_\beta$ if and only if $\I(\alpha)$ is a dimension vector of $\river$. If $\alpha$ is a dimension vector then the Embedding Theorem \ref{thm:embedding-semi-invariants} yields $(\ref{eqn-embed-1})$.  Otherwise, both quantities in $(\ref{eqn-embed-1})$ are equal to zero by Proposition \ref{wt-dim-vector-prop}.

Finally, it follows from $(\ref{eqn-mult-wt-semi-inv})$, $(\ref{eqn-embed-1})$, and Proposition \ref{main-compute-semi-inv-LR-prop} that
$$
K^{\beta}_{\unlm} = \sum_{\substack{\mu\\ \ell(\mu)\leq nd}} c^\mu_{\lambda(x_1),\ldots, \lambda(x_m), \underbrace{(f^d),\ldots, (f^d)}_{n \text{ times }}} \cdot c^\mu_{\lm(y_1),\ldots, \lm(y_\ell),(f^{nd})},
$$
where $d=\sum_{i\in [m]} \beta(x_i)$ and $f=\sum_{i \in [m]} \lambda_1(x_i)$. This completes the proof.
\end{proof}

\begin{remark}\phantomsection\label{rmk-direct-compute-vs-river-quiver}
As indicated in formula $(\ref{eqn-mult-wt-semi-inv})$ above, one can compute $K^{\beta}_{\unlm}$ directly (without embedding $\Q_\beta$ into $\river$) in terms of Littlewood-Richardson coefficients. The problem with this direct approach is that it computes $K^{\beta}_{\unlm}$ as a sum over $l  m  n$ variable partitions $\mu_{i,j}^{(r)}$,  $i\in [m], j \in [l], r \in [n[$, where each term of the sum is a product of $m l$ multiple Littlewood-Richardson coefficients. The result is very difficult to work with, making our approach based on quiver exceptional sequences and the quiver $\river$ essential for our purposes. 
\end{remark}

As a consequence of Theorem \ref{main-formula-mult-thm}, we obtain the following interesting combinatorial identity.

\begin{corollary} Let $d$ and $n$ be two positive integers and let $\lambda = (\lambda_1,\ldots, \lambda_d)$ and $\nu = (\nu_1,\ldots, \nu_d)$ be two partitions of length at most $d$. Then
\begin{equation} \label{formula-semi-inv-n-Kronecker}
\sum_{\mu(1), \ldots, \mu(n)} c^{\lambda}_{\mu(1), \ldots, \mu(n)} \cdot  c^{\nu}_{\mu(1), \ldots, \mu(n)}=c^{\nu+(\lambda_1^{nd})}_{\lambda, \underbrace{ (\lambda_1^d),\ldots, (\lambda_1^d)}_{n \text{ times }}},
\end{equation}
where the sum on the left hand side is over all partitions $\mu(1), \ldots, \mu(n)$ of length at most $d$.
\end{corollary}

\begin{proof}
Let $\Q$ be the $n$-Kronecker quiver 
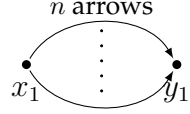
\[
\begin{tikzpicture}[point/.style={shape=circle, fill=black, scale=.3pt,outer sep=3pt},>=latex]
   \node[point,label={below:$x_1$}] (1) at (-2.5,0) {};
   \node[point,label={below:$y_1$}] (4) at (-.5,0) {};

         \draw[decorate sep={.25mm}{2mm},fill] (-1.5,.45)--(-1.5,-.45);       
           
    \path[->]  
  (1) edge  [bend left=60] (4);
  
      \path[->]  
  (1) edge  [bend right=60] (4);
  
    \node at (-1.5,.75) { \footnotesize $n$ arrows}; 
    \end{tikzpicture}
    \]
and let $\beta = (d,d)$.  Then it follows from formula $(\ref{eqn-mult-wt-semi-inv})$ that the left hand side of $(\ref{formula-semi-inv-n-Kronecker})$ is precisely $K^{\beta}_{(\lambda, -\nu)}$. The identity now follows from Theorem \ref{main-formula-mult-thm} and Remark \ref{main-compute-one-sink-rmk}.
\end{proof}

\begin{rmk} If $\lambda=\nu=(x^d)$ for some non-negative integer $x$, then the left hand side of $(\ref{formula-semi-inv-n-Kronecker})$ is precisely $\dim \SI(\Q, (d,d))_{(x,-x)}$ where $\Q$ is the $n$-Kronecker quiver. In this case, our corollary shows that $\dim \SI(\Q, (d,d))_{(x, -x)}$ is a parabolic Kostka coefficient.
\end{rmk}

\subsection{Hive-type polytopes for quiver multiplicities} \label{hives-semi-invs-section}
Our goal in this subsection is to find a polytopal description for constants of the form
\[
K_{\unlm, f}(d,n) := \sum_{\substack{\mu \\ \ell(\mu) \leq nd}} c^\mu_{\lm(x_1),\ldots, \lm(x_m), \underbrace{(f^d),\ldots, (f^d)}_{n \text{ times}}} \cdot c^\mu_{\lm(y_1),\ldots, \lm(y_\ell),(f^{nd})}
\]
where $f, d, \ell, m,n$ are fixed positive integers and $\lm(x_1),\ldots, \lm(x_m), \lm(y_1),\ldots, \lm(y_\ell)$ are fixed partitions such that $\sum_{i=1}^m |\lm(x_i)| = \sum_{j=1}^\ell |\lm(y_j)|$. As we have seen in Section \ref{subsection:hives}, these types of structure constants occur as our multiplicities $K^{\beta}_{\unlm}$.

We begin by applying Proposition \ref{Zelevinsky-trick} to the terms of the sum in the definition of $K_{\unlm, f}(d,n)$. To this end,  we first extend each of the partitions $\lm(x_i)$, $\lm(y_j)$, $(f^d)$, and $(f^{nd})$ by adding zero parts so that their length is at most $\sum_{i=1}^m \ell(\lm(x_i))+\sum_{j=1}^\ell \ell(\lm(y_j))+nd$. Using $(\ref{Zel-trick-eqns})$, we next construct the partitions $\gamma(1) \subset \gamma(2)$ and $\gamma(3)\subset \gamma(4)$ such that 
\begin{equation}\label{partitions-gamma-from-Zele-trick}
c^\mu_{\lm(x_1),\ldots, \lm(x_m), \underbrace{(f^d),\ldots, (f^d)}_{n \text{ times}}}=c^{\gamma(2)}_{\gamma(1),\mu} \text{~and~} c^\mu_{\lm(y_1),\ldots, \lm(\lm(y_\ell),(f^{nd})}=c^{\gamma(4)}_{\mu,\gamma(3)}.
\end{equation}
Note that $\gamma(1), \gamma(2), \gamma(3)$, and $\gamma(4)$ have at most $N$ parts where 
$$
N:=(m+n+l+1)\left( \sum_{i=1}^m \ell(\lm(x_i))+\sum_{j=1}^\ell \ell(\lm(y_j))+nd \right).
$$

It now follows from Proposition \ref{Zelevinsky-trick} that 
\[
K_{\unlm, f}(d,n) = \sum_{\substack{\mu \\ \ell(\mu) \leq nd}} c^{\gamma(2)}_{\gamma(1),\mu} \cdot c^{\gamma(4)}_{\mu,\gamma(3)}.
\]

Let us now consider the polytope obtained by gluing two hive polytopes as follows:

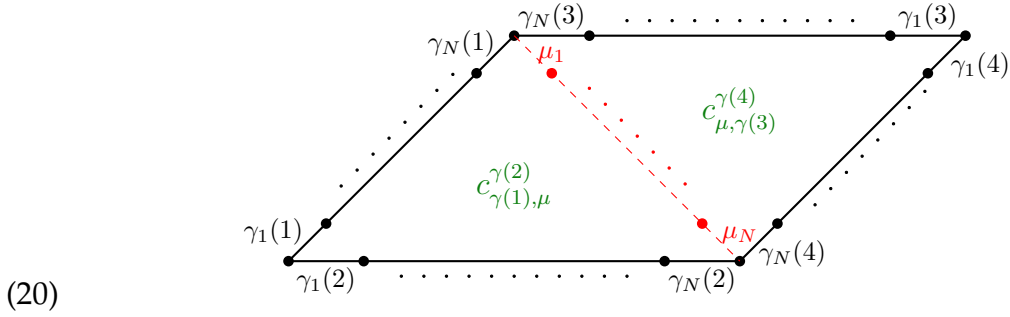
\begin{equation}\label{our-polytope-pic}
\begin{tikzpicture}
\coordinate (a) at (0,0);
\node[font=\footnotesize] (nu11) at (-.2,.4) {$\gamma_1(1)$};
\coordinate (b) at (.5,.5);
\coordinate (c) at (2.5,2.5);
\node[font=\footnotesize] (nuN1) at (2.3,2.9) {$\gamma_N(1)$};
\coordinate (d) at (3,3);
\coordinate (e) at (4,3);
\node[font=\footnotesize] (nuN3) at (3.5,3.25) {$\gamma_N(3)$};
\coordinate (f) at (8,3); 
\coordinate (g) at (9,3);
\node[font=\footnotesize] (nu13) at (8.5,3.25) {$\gamma_1(3)$};
\coordinate (h) at (8.5,2.5);
\node[font=\footnotesize] (nu14) at (9.2,2.6) {$\gamma_1(4)$};
\coordinate (i) at (6.5,.5);
\coordinate (j) at (6,0);
\node[font=\footnotesize] (nuN4) at (6.7,.1) {$\gamma_N(4)$};
\coordinate (k) at (5,0);
\node[font=\footnotesize] (nuN2) at (5.5,-.25) {$\gamma_N(2)$};
\coordinate (l) at (1,0);
\node[font=\footnotesize] (nu12) at (.5,-.25) {$\gamma_1(2)$};
\draw[black,thick] (a)--(b)--(c)--(d)--(e)--(f)--(g)--(h)--(i)--(j)--(k)--(l)--cycle;

\foreach \i in {a,b,c,d,e,f,g,h,i,j,k,l}
  \fill (\i) circle (2pt);

  \draw[red,dashed] (3,3) -- (6,0);
  \fill[red] (3.5,2.5) circle (2pt);
  \fill[red] (5.5,.5) circle (2pt);
  
\node[font=\footnotesize] (m1) at (3.5,2.75) {$\textcolor{red}{\mu_1}$};
\node[font=\footnotesize] (mN) at (6,.35) {$\textcolor{red}{\mu_N}$};

   \coordinate (c1) at (3,1);
    \draw (c1) node[above] {} node {$\textcolor{dartmouthgreen}{c^{\gamma(2)}_{\gamma(1),\mu}}$};
  \coordinate (c2) at (6,2);
     \draw (c2) node[above] {} node {$\textcolor{dartmouthgreen}{c^{\gamma(4)}_{\mu,\gamma(3)}}$};

\draw[decorate sep={.3mm}{3mm},fill] (1.5,-.2) -- (4.5,-.2);
\draw[decorate sep={.3mm}{3mm},fill] (.7,1) -- (2.2,2.5);
\draw[decorate sep={.3mm}{3mm},fill] (4.5,3.2) -- (7.5,3.2);
\draw[decorate sep={.3mm}{3mm},fill] (7,.75) -- (8.5,2.3);
\draw[red, decorate sep={.3mm}{3mm},fill] (4,2.3) -- (5.3,1);

\end{tikzpicture}
\end{equation}

\noindent
Specifically, we define $\p_{\unlm, f}(d,n)$ to be the polytope consisting of all tuples of non-negative numbers $(x_{i,j}, y_{i,j}, t_{i,j}, \widetilde{x}_{i,j}, \widetilde{y}_{i,j}, \widetilde{t}_{i,j})$ such that 
\begin{enumerate}
\item $x_{i,0} = \gamma_{i+1}(1), \; t_{0,k} = \gamma_{k+1}(2), \; \forall i,k \in \{0,\ldots, N-1\}$;

\item $y_{j,N-1-j} = \widetilde{y}_{j,N-1-j}, \; \forall j \in \{0,\ldots, N-1\}$;

\item $y_{nd+j, N-1-(nd+j)}=0$, $\forall j \in [N-nd]$;

\item $\widetilde{x}_{i,0} = \gamma_{i+1}(3), \; \widetilde{t}_{0,k} = \gamma_{k+1}(4), \; \forall i,k \in \{0,\ldots, N-1\}$;
\item $\sum_{j=0}^{N-1} y_{j, N-1-j} = |\gamma(2)| - |\gamma(1)| = |\gamma(4)|-|\gamma(3)|$;
\item $(x_{i,j},y_{i,j},t_{i,j})$ and $(\widetilde{x}_{i,j}, \widetilde{y}_{i,j}, \widetilde{t}_{i,j})$ are $N$-hives.
\end{enumerate}

\smallskip
\noindent
It follows from Theorem \ref{LR-hives-KnuTao-thm} that the number of lattice points of $\p_{\unlm, f}(d,n)$ is 
\[
\sum_{\mu} c^{\gamma(2)}_{\gamma(1),\mu} \cdot c^{\gamma(4)}_{\mu,\gamma(3)},
\]
where the sum is over all partitions $\mu$ with $\ell(\mu) \leq N$ whose last $N-nd$ parts are zero. Thus, we get that

\begin{equation}\label{lattice-points-poly-eqn-1}
K_{\unlm, f}(d, n)=\text{~the number of lattice points of~}\p_{\unlm, f}(d,n).
\end{equation}

\begin{rmk} \label{comb-linear-program-rmk} The linear inequalities defining $\p_{\unlm, f}(d,n)$ can be written in the form of an  integer linear program
$$
A\cdot \mathbf{x} \leq \mathbf{b},
$$
where the entries of $A$ are $0$, $1$, and $-1$, and the  entries of $\mathbf{b}$ are homogeneous linear integral forms in the parts of the partitions $\lm(x_i)$, $\lm(y_j)$, and $f$. This is a combinatorial linear program in the sense of Tardos \cite{Tar86}.
\end{rmk}

\subsection{The polytope $\p_{\unlm}$ from Theorem \ref{main-thm}}\label{our-polytope-sec} Let $\Q$ be the $n$-complete bipartite quiver with source vertices $x_1, \ldots, x_m$, and sink vertices $y_1, \ldots, y_\ell$. Let $\beta$ be a sincere dimension vector of $\Q$ and let $(\Q_\beta, \widetilde{\beta})$ be the flag-extension of $(\Q, \beta)$.

\begin{definition}[\textbf{The polytope $\p_{\unlm}$}] \label{defn-our-polytope}
Let $\unlm=(\lm(x_i), -\lm(y_j))_{i \in [m], j \in [\ell] }$ be a tuple of weakly decreasing sequences with $\lm(x_i)$ a partition of length at most $\beta(x_i)$ and $\lm(y_j)$ a partition of length at most $\beta(y_j)$ such that
$$
\sum_{i=1}^m |\lambda(x_i)|=\sum_{j=1}^{\ell} |\lambda(y_j)|.
$$

We define
\begin{equation}
\p_{\unlm}:=\p_{\unlm, f}(d,n),
\end{equation}
where 
$$
f:=\sum_{i \in [m]} \lambda_1(x_i) \text{~and~} d:=\sum_{i \in [m]} \beta(x_i). 
$$

\end{definition}

As a direct consequence of Theorem \ref{main-formula-mult-thm} and the Saturation Property of Derksen and Weyman, we obtain the following polytopal description of the multiplicities $K^{\beta}_{\unlm}$. 

\begin{prop} \label{prop-main-thm-part-1} Keep the same notations as above. Then
\begin{equation}\label{lattice-points-our-polytope-eqn2}
K^{\beta}_{\unlm}=\sum_{\substack{\mu \\ \ell(\mu) \leq nd}} c^{\gamma(2)}_{\gamma(1),\mu} \cdot c^{\gamma(4)}_{\mu,\gamma(3)}=\text{~the number of lattice points of~}\p_{\unlm},
\end{equation}
where $\gamma(1), \gamma(2), \gamma(3), \gamma(4)$ are obtained from $\unlm$ via $(\ref{partitions-gamma-from-Zele-trick})$. Furthermore,
\begin{equation}\label{our-mult-non-empty-polytope-eqn}
K^{\beta}_{\unlm} \neq 0 \Longleftrightarrow \p_{\unlm} \neq \emptyset.
\end{equation}
\end{prop}

\begin{proof} The first part, formula $(\ref{lattice-points-our-polytope-eqn2})$, follows at once from Theorem \ref{main-formula-mult-thm} and $(\ref{lattice-points-poly-eqn-1})$.

When it comes to $(\ref{our-mult-non-empty-polytope-eqn})$, the implication ``$\Longrightarrow$" is obvious. For the other implication,  assume that $\p_{\unlm} \neq \emptyset$ and let $\mathbf{v}$ be one of its vertices. Then $\mathbf{v}$ must have rational coefficients, and therefore $r\cdot \mathbf{v}$ is a lattice point of $\p_{r \unlm}$ for some positive integer $r$, and thus $K^{\beta}_{r \unlm} \neq 0$ by $(\ref{lattice-points-our-polytope-eqn2})$. But $K^{\beta}_{r \unlm}$ can also be expressed via $(\ref{eqn-mult-wt-semi-inv})$ as the dimension of a weight spaces of quiver semi-invariants  of the form $\dim \SI(\Q_\beta, \widetilde{\beta})_{r \widetilde{\sigma}_{\unlm}}$. It now follows from the Saturation Property stated in Theorem \ref{King-criterion} that $ K^{\beta}_{\unlm}$, which can be expressed as $\dim \SI(\Q_\beta, \widetilde{\beta})_{\widetilde{\sigma}_{\unlm}}$, is also non-zero. 
\end{proof}

\section{Moment cones for quivers and the proof of Theorem \ref{main-thm}}\label{moment-cones-proof-thm-1-sec}
Let $Q=(Q_0,Q_1,t,h)$ be a connected acyclic quiver and $\beta \in \ZZ_{>0}^{Q_0}$ be a sincere dimension vector of $Q$. If $U(\beta(x))$ is the group of $\beta(x) \times \beta(x)$ unitary matrices for every $x \in Q_0$, then 
\[
U(\beta) := \prod_{x \in Q_0} U(\beta(x))
\]
is a maximal compact subgroup of $\GL(\beta)$. The conjugation action of $U(\beta)$ on $\rep(Q,\beta)$ is Hamiltonian with the moment map given by 
\begin{align*}
\phi: & \rep(Q,\beta) \to \Herm(\beta)\\
& W \mapsto \phi(W) := \left( \sum_{\substack{a \in Q_1 \\ ta=x}} W(a)^* \cdot W(a) - \sum_{\substack{a \in Q_1 \\ ha=x}} W(a) \cdot W(a)^*\right)_{x \in Q_0}
\end{align*}
where $\Herm(\beta) := \prod_{x \in Q_0} \Herm(\beta(x))$ with $\Herm(\beta(x))$ being the space of $\beta(x) \times \beta(x)$ Hermitian matrices for every $x \in Q_0$ and  $W(a)^*$ denotes the adjoint of the complex matrix $W(a)$, \emph{i.e.}, $W(a)^*$ is the transpose of the conjugate of $W(a)$. The moment cone corresponding to this moment map is $\Delta(Q,\beta)$, which is a rational convex polyhedral cone (see \cite[Theorem 4.9]{Sj}) and  can be viewed as the cone over the moment polytope of the projectivization of $\rep(Q,\beta)$ (see \cite[Corollary 4.11]{Sj}). A more in-depth description of $\Delta(W, \beta)$ can be found in \cite{Bal-Ver-Wal-2023}. Nonetheless, this description does not provide a strongly polynomial time algorithm for testing membership in $\Delta(Q, \beta)$.

If $\lambda=(\lambda_1, \ldots, \lambda_N)$ is a weakly decreasing sequence of real numbers, then the weakly decreasing sequence $(-\lambda_N, \ldots, -\lambda_1)$ will be denoted by $-\lambda$. 

\begin{example}\label{Kly-cone-ex} Let $Q=\bullet \rightarrow \bullet \leftarrow \bullet$ and $\beta =(r,r,r)$. Then $\Delta(Q, \beta)$ consists of all triples $(\lambda(1), \lambda(2), -\lambda(3))$ with each $\lambda(i)$ a weakly decreasing sequence of $r$ (non-negative) real numbers for which there are positive semi-definite $r \times r$ Hermitian matrices $H(1), H(2)$, and $ H(3)$ with spectra $\lambda(1), \lambda(2)$,  and $\lambda(3)$, respectively, and $H(3)=H(1)+H(2)$.

Recall that the Klyachko cone, denoted by $\K(r)$, consists of all triples $(\lambda(1), \lambda(2), \lambda(3))$ of weakly decreasing sequences of $r$ real numbers for which there are $r \times r$ Hermitian matrices $H(1), H(2)$, and $ H(3)$ with spectra $\lambda(1), \lambda(2)$,  and $\lambda(3)$, respectively, and $H(3)=H(1)+H(2)$.

Now, let $(\lambda(1), \lambda(2), \lambda(3))$ be a triple of weakly decreasing sequences of $r$ real numbers, and consider the following sequences of non-negative real numbers:
\begin{align*}
&\widetilde{\lambda}(1):=(\lambda_1(1)-\lambda_r(1), \ldots, \lambda_{r-1}(1)-\lambda_r(1), 0),\\
& \widetilde{\lambda}(2):=(\lambda_1(2)-\lambda_r(2), \ldots, \lambda_{r-1}(2)-\lambda_r(2), 0), \\
&\widetilde{\lambda}(3):=(\lambda_1(3)-(\lambda_r(1)+\lambda_r(2)), \ldots, \lambda_r(3)-(\lambda_r(1)+\lambda_r(2))).
\end{align*}

It is now immediate to see that
$$
(\lambda(1), \lambda(2), \lambda(3)) \in \K(r) \Longleftrightarrow (\widetilde{\lambda}(1), \widetilde{\lambda}(2),-\widetilde{\lambda}(3)) \in \Delta(Q, \beta).
$$
\end{example}

We next explain how to view $\Delta(Q,\beta)$ as the cone of effective weights associated to a different quiver. While this result holds for general quivers (see \cite{Ber-Rei-2022}), we will focus in what follows on bipartite quivers since this suffices for our purposes. 

Assume that $Q$ is a bipartite quiver (not necessarily $n$-complete) with source vertices $x_1, \ldots, x_m$, and sink vertices $y_1, \ldots, y_\ell$, and let $(Q_\beta, \widetilde{\beta})$ be its flag-extension. (Note that we orient our flags slightly differently than in \cite{Ber-Rei-2022}.) 

Let $\underline{\lambda} = (\lambda(x_i), -\lambda(y_j))_{i \in [m], j \in [\ell]}$ be a tuple of sequences with $\lambda(x)$ a weakly decreasing sequence of $\beta(x)$ real numbers for every vertex $x \in Q_0$.  

\begin{remark} \label{spectra-delta-rmk} It is immediate to see that $\underline{\lambda}$ belongs to $\Delta(Q, \beta)$ if and only if there is a representation $W \in \rep(Q, \beta)$ such that 
\begin{enumerate}
\item the spectrum of the Hermitian matrix $
\sum_{\substack{a \in Q_1 \\ ta=x_i}} W(a)^* \cdot W(a)$ is $\lambda(x_i)$ for every $i \in [m]$;
\smallskip

\item the spectrum of the Hermitian matrix $\sum_{\substack{a \in Q_1 \\ ha=y_j}} W(a) \cdot W(a)^*$ is $\lambda(y_j)$ for every $j \in [\ell]$.
\end{enumerate}

\noindent
This shows that a necessary condition for $\underline{\lambda}$ to belong to $\Delta(Q, \beta)$ is that $\lambda(x_i)$, $i \in [m]$, and $\lambda(y_j)$, $j \in [\ell]$, are non-negative sequences.
\end{remark}

Now let $\widetilde{\sigma}_{\underline{\lambda}} \in \RR^{(Q_\beta)_0}$ be the real weight defined as follows: If $x$ is a source vertex of $Q$, the values of $\widetilde{\sigma}_{\underline{\lambda}}$ along the $\beta(x)$ vertices of the flag 
\[
\F(x) :  \bullet \rightarrow \bullet \rightarrow \cdots \bullet \rightarrow \hspace{-6pt} \begin{array}{c}\\ \bullet \\ x \end{array}
\]
are 
$$
\lambda_1(x) - \lambda_2(x), \, \ldots, \, \lambda_{\beta(x)-1}(x) - \lambda_{\beta(x)}(x), \, \lambda_{\beta(x)}(x).
$$

If, instead, $y$ is a sink vertex of $Q$, the values of $\widetilde{\sigma}_{\underline{\lambda}}$ along the $\beta(y)$ vertices of the flag
\[
\F(x) : \begin{array}{c}\\ \bullet \\ y \end{array} \hspace{-6pt} \rightarrow \bullet \rightarrow \cdots \bullet \rightarrow \bullet
\]
are 
$$
-\lambda_{\beta(y)}(y), \,  \lambda_{\beta(y)}(y) - \lambda_{\beta(y)-1}(y), \, \ldots, \, \lambda_2(y) - \lambda_1(y).
$$

\smallskip
\begin{prop}[compare to \cite{Ber-Rei-2022}] \label{moment-cone-effecive-cone-prop} Let $Q$ be a bipartite quiver and $\beta$ a sincere dimension vector of $Q$. Let $T$ be the function from the set of all tuples $\underline{\lambda} = (\lambda(x_i), -\lambda(y_j))_{i \in [m], j \in [\ell]}$ as above to $\RR^{(Q_\beta)_0}$ defined by $T(\lambda) = \widetilde{\sigma}_{\underline{\lambda}}$. Then
$$
T(\Delta(Q, \beta))=\Eff(Q_\beta, \widetilde{\beta}),
$$
and $T$ is an isomorphism of rational convex polyhedral cones.
\end{prop}

To prove this result, we require the following very useful lemma.

\begin{lemma}[see {\cite[Sec. 3.4]{CB}}]
\label{lemma:moment-matrices}
Let $\sigma(1), \ldots, \sigma(N-1)$ be non-negative real numbers. Then the following are equivalent:
			\begin{enumerate}[(a)]
				\item There exist matrices $W_i \in \CC^{(i+1)\times i}$, $1 \leq i \leq N-1$, such that 
				\begin{align*}
				W_i^* \cdot W_i - W_{i-1} \cdot W_{i-1}^* &= \sigma(i) \cdot \Id_{\CC^i} \; \text{ for } 2 \leq i \leq N-1,\\
				W_1^* \cdot W_1 &=\sigma(1).
				\end{align*}
				\item There exists an $N \times N$ Hermitian matrix $H$ $\left(= W_{N-1} \cdot W_{N-1}^* \right)$ with eigenvalues 
				\[
				\gamma(i) = \sum_{i \leq j \leq N-1} \sigma(j), \; \forall \; 1 \leq i \leq N-1,
				\]
				and $\gamma(N) = 0$.
			\end{enumerate}
\end{lemma}

We are now ready to prove Proposition \ref{moment-cone-effecive-cone-prop}.
 
\begin{proof}[Proof of Proposition \ref{moment-cone-effecive-cone-prop}]
Let $\underline{\lambda} = (\lambda(x_i), -\lambda(y_j))_{i \in [m], j \in [\ell]}$ be a tuple of sequences with $\lambda(x)$ a weakly decreasing sequence of $\beta(x)$ real numbers for every vertex $x$ of $Q$. From Remark \ref{spectra-delta-rmk} and Lemma \ref{lemma:moment-matrices} we obtain that $\underline{\lambda} \in \Delta(Q, \beta)$ if and only if there exists $\widetilde{W} \in \rep(Q_\beta, \widetilde{\beta})$ such that 
	\[
	\sum_{\substack{a \in Q_1 \\ ta=x}} \widetilde{W}(a)^* \cdot \widetilde{W}(a) - \sum_{\substack{a \in Q_1 \\ ha=x}} \widetilde{W}(a) \cdot \widetilde{W}(a)^* = \widetilde{\sigma}_{\underline{\lambda}}(x) \cdot \Id_{\widetilde{\beta}(x)} \; \forall x \in (Q_\beta)_0.
	\]

It now follows from Theorem \ref{King-criterion} that $T\left(\Delta(Q,\beta) \cap \ZZ^{(Q_\beta)_0}\right) \subseteq \Eff(Q_\beta, \widetilde{\beta}) \cap \ZZ^{(Q_\beta)_0}$. To prove the other inclusion, let $\widetilde{\sigma} \in \Eff(Q_\beta,\widetilde{\beta})$ be any effective weight. Then $\widetilde{\sigma}$ is  non-negative/non-positive along the vertices of the flag $\F(x)$ if $x$ is a source/sink of $Q$ by Lemma \ref{eff-wts-cone-flags-lemma}. For any such $\widetilde{\sigma}$, consider the partitions
\[
\lambda_{\widetilde{\sigma}}(x) : = \begin{cases} 
\left( \dsp \sum_{i \leq j \leq \beta(x)} \widetilde{\sigma}(j) \right)_{i \in [\beta(x)]} & \mbox{ if } x \text{ is a source} \\
\\
\left( \dsp -\sum_{i \leq j \leq \beta(x)} \widetilde{\sigma}(j)\right)_{i \in [\beta(x)]} & \mbox{ if } x \text{ is a sink},
\end{cases}
\]
where $\widetilde{\sigma}(k)$ denotes the value of $\widetilde{\sigma}$ at the $k^{th}$ vertex of the flag $\F(x)$ as we traverse the flag from left/right to right/left for any source/sink vertex $x \in Q_0$ and $k \in [\beta(x)]$. Then, using Lemma \ref{lemma:moment-matrices} once again, we get that $\underline{\lambda}_{\widetilde{\sigma}}:=(\lambda_{\widetilde{\sigma}}(x_i), -\lambda_{\widetilde{\sigma}}(y_j))_{i \in [m], j\in [\ell]}$ belongs to $\Delta(Q, \beta)$ and
$$
T(\underline{\lambda}_{\widetilde{\sigma}}) = \widetilde{\sigma}.
$$
This shows that 
$T(\Delta(Q,\beta)) \cap \ZZ^{(Q_\beta)_0} = \Eff(Q_\beta, \widetilde{\beta}) \cap \ZZ^{(Q_\beta)_0}$, which implies the claim of the proposition since $\Delta(Q,\beta)$ and $\Eff(Q_\beta,\widetilde{\beta})$ are both rational convex polyhedral cones.
\end{proof}

Finally, we are ready to prove our main result.

\begin{proof}[\textbf{Proof of Theorem \ref{main-thm}}] $(1)$ This part is proved in Proposition \ref{prop-main-thm-part-1}.

\smallskip
\noindent
$(2)$ Let $\underline{\lambda}=(\lambda(x_i), -\lambda(y_j))_{i \in [m], j \in [\ell]}$ be a tuple of sequences with $\lambda(x)$ a weakly decreasing sequence of $\beta(x)$ integers for every $x \in \Q_0$.

We assume that the $\lambda(x_i)$ and $\lambda(y_j)$ are partitions since otherwise we know that $\underline{\lambda} \notin \Delta(\Q, \beta)$ by Remark \ref{spectra-delta-rmk}. It now follows from part (1) and Propositions \ref{moment-cone-effecive-cone-prop} and \ref{prop-main-thm-part-1} that
$$
\underline{\lambda} \in \Delta(\Q, \beta) \Longleftrightarrow K^{\beta}_{\unlm} \neq 0 \Longleftrightarrow \p_{\unlm} \neq \emptyset.
$$

Since $\p_{\unlm}$ can be described as a combinatorial linear program (see Remark \ref{comb-linear-program-rmk}), deciding whether $\underline{\lambda}$ belongs to $\Delta(\Q, \beta)$ can be done in strongly polynomial time using Tardos' \cite{Tar86} combinatorial linear programming algorithm.
\end{proof}

\begin{rmk}\label{rmk-non-natural-input-form} Let $\beta=(\beta(x))_{x \in \Q_0 } \in \ZZ_{>0}^{\Q_0}$ be a sincere dimension vector and let 
\[
\sigma=(\sigma(x_i), \sigma(y_j))_{i \in [m], j \in [\ell]} \in \ZZ^{\Q_0}
\]
be an integral stability weight for $\Q$ with $\sigma(x_i) \geq 0$, $\forall i \in [m]$, and $\sigma(y_j) \leq 0$, $\forall j \in [\ell]$. 

Set $\lambda(x_i):=(\underbrace{\sigma(x_i), \ldots, \sigma(x_i))}_{\beta(x_i)}$ for every $i \in [m]$, and $\lambda(y_j):=(\underbrace{-\sigma(y_j), \ldots, -\sigma(y_j))}_{\beta(y_j)}$ for every $j \in [\ell]$, and let 
$$
\underline{\lambda}_{\sigma}=(\lambda(x_i), -\lambda(y_j))_{i \in [m], j \in [\ell]}.
$$
Then it follows from Theorem \ref{King-criterion} that
$$
\sigma \in \Eff(\Q, \beta) \Longleftrightarrow \underline{\lambda}_{\sigma} \in \Delta(\Q, \beta).
$$
Thus, if the input in Problem \ref{gen-semi-stab-problem}  is specified as $\underline{\lambda}_{\sigma}$, then Theorem \ref{main-thm} implies a strongly polynomial time algorithm for the generic semi-stability problem.  
\end{rmk}

\subsection*{Acknowledgment} The authors would like to thank Harm Derksen, Cole Franks, and Visu Makam for their helpful conversations on the subject of the paper. We are also grateful to Mich{\`e}le Vergne and Michael Walter for their enlightening discussions on the paper. Additionally, we extend our thanks to the anonymous referee for their valuable comments and suggestions, which significantly improved the exposition of this work.

C. Chindris is supported by Simons Foundation grant $\# 711639$.

\end{document}